\documentclass[a4paper]{amsart}

\usepackage{amsmath, amsfonts, amsthm , amssymb, amscd}
\usepackage{systeme} 

\usepackage{graphicx}
\usepackage{pst-grad} % For gradients and colours
\usepackage{tikz-cd}
\usepackage{tikz}
\usetikzlibrary{matrix,arrows,decorations.pathmorphing}
\usetikzlibrary{shapes.geometric}
\usepackage{url}
\usepackage{pstricks}
 \newtheorem{thm}{Theorem}[section]
 \newtheorem{prop}[thm]{Proposition}
 \newtheorem{lem}[thm]{Lemma}
 \newtheorem{cor}[thm]{Corollary}
 \newtheorem{exm}[thm]{Example}
 \newtheorem{rem}[thm]{Remark}
 \newtheorem{que}[thm]{Question}
 \newtheorem{dfn}{Definition}[section]
\numberwithin{equation}{section}

\title{Geometrical realisations of the simple permutoassociahedron 
by Minkowski sums}
\author{Jelena Ivanovi\' c}
\date{}

\usepackage{fancyhdr, blindtext}
\newcommand{\changefont}{
    \fontsize{7}{12}\selectfont
}

 {\changefont \slshape \rightmark} 
 {\changefont \slshape \rightmark}
{\changefont}
\begin{document}

\begin{abstract}
This paper introduces a family of $n$-polytopes, $PA_{n,c}$ which is a geometrical realisation of simple permutoassociahedra. It has significant importance serving as a topological proof of Mac Lane's coherence.
Polytopes in this family are defined as Minkowski sums of certain polytopes such that every summand produces exactly one truncation of the permutohedron, i.e. yields to the appropriate facet of the resulting sum.
%the sum of a partial sum $S$ and any other summand in the presentation is combinatorially equivalent to the polytope which could be obtained from $S$ by truncating in its proper face. 
Additionally, it leads to the correlation between Minkowski sums and truncations, which gives a general procedure for similar geometrical realisation of a wider class of polytopes.

\vspace{.3cm}


\vspace{.5ex}

\noindent {\small {\it Keywords$\,$}: coherence, simple polytopes, geometrical realisation, Minkowski sum}
\end{abstract}

\maketitle

\section{introduction} \label{s_uvod}
A convex \emph{polytope} $P$ can be defined as a bounded intersection of a finitely many halfspaces. More precisely, it is a bounded solution set of a finite system of linear inequalities: 
\[
P=P(A,b):=\{x \in \mathbf{R}^n \mid \langle a_i,x\rangle\geqslant b_i,  \;   \; 1 \leqslant i \leqslant m\},
\]
where $A \in \mathbf{R}^{m \times n}$ is a real matrix with rows $a_i$, and $b\in \mathbf{R}^m$ is a real vector with entries $b_i$. Here, boundedness means that there is a constant $N$ such that $\Vert x \Vert \leqslant N$ holds for all $x \in P$. Also, convex polytope can be defined as a convex hull of a finite set of points in $\mathbf{R}^n$. Although equivalent (\cite[Theorem~1.1]{Z95}), these two definitions are essentially different from an algorithmic point of view. Through this paper, we use both. Since we consider only convex polytopes, we omit the word ``convex''.

The \emph{dimension} of a polytope $P$, denoted by $\dim(P)$, is the dimension of its affine hull. A polytope of dimension $d \leqslant n$ is written as \emph{$d$-polytope}. For a hyperplane $H$, the intersection $P\cap H$ is  called  a \emph{face}
of $P$ when $P$ lies in one of the halfspaces determined by $H$. If $P\cap H\neq \emptyset$, $H$ is a \emph{supporting hyperplane}.
We say that a face $F$ of $P$ is parallel to the given hyperplane $\pi$ when there is a hyperplane $H$ parallel to $\pi$ which defines $F$. Faces of dimensions 0, 1, and $d-1$ are called vertices, edges, and facets, respectively.
The sets of vertices and facets is denoted by $\mathcal{V}(P)$ and $\mathcal{F}(P)$, respectively.
A $d$-polytope is called
\textit{simple}, if each of its vertices belongs to exactly
$d$ facets (equivalently, to exactly
$d$ edges).
For the polytope $P=P(A,b)$, the halfspace defined by $i$th inequality $\langle a_i,x\rangle\geqslant b_i$ is called \emph{facet-defining}, when $\{x \in P \mid \langle a_i,x\rangle = b_i\}$ is a facet. Hence, $-a_i$, an \emph{outward normal vector} to that halfspace, is an outward normal vector to that facet.

For an equation that corresponds to the hyperplane $\pi$,
the \emph{halfspaces} $\pi^{\geqslant}$ and $\pi^{\leqslant}$ are defined as $\pi$, save that ``$=$'' is
replaced by ``$\geqslant$'' and ``$\leqslant$'', respectively.
For an arbitrary polytope $P$, $\pi^\geqslant$ is \emph{beneath} a vertex $V \in P$ when $V$ belongs to $\pi^>$ and also, we say that $\pi^\geqslant$ is \emph{beyond} $V$ when $V$ does not belong to $\pi^\geqslant$. A \textit{truncation} tr$_FP$ of $P$ in its proper face $F$ is a polytope $P \cap\pi^\geqslant$, where $\pi^\geqslant$ is  beneath every vertex not contained in $F$ and beyond every vertex contained in $F$. This truncation is \emph{parallel} when $F$ is parallel to $\pi$. In this paper, we assume that all truncations are in the faces that are not facets.

A simple polytope named
\textit{permutoassociahedron} belongs to the family that generalises a well known family of polytopes called \textit{nestohedra}, i.e.\ \textit{hypergraph polytopes}
%, which were widely investigated by many different authors
(see \cite{FMP94}, \cite{P09} or \cite{PRW08}).
%: Fulton and MacPherson \cite{FMP94}, De Concini and Procesi
%\cite{DeCP95}, Stasheff and Shnider \cite{S97}, Gaiffi
%\cite{Ga03} \cite{Ga04}, Feichtner and Kozlov \cite{FK04},
%Feichtner and Sturmfels \cite{FS05}, Carr and Devadoss
%\cite{CD06}, Postnikov, Reiner and Williams \cite{PRW08},
%Postnikov \cite{P09}, Do\v sen and Petri\' c \cite{DP10}, and many others.
%Nowadays,
%Remarkable properties of the nestohedra find a numerous applications in various fields of mathematics, especially in algebra, combinatorics, geometry, topology and logic.
Nestohedra appear in many fields of mathematics, especially in algebra, combinatorics, geometry, topology and logic.
Roughly speaking, we can understand this family as polytopes that can be obtained by truncations in the vertices, edges and other proper faces of $d$-\textit{simplex}. The recipe that prescribes which faces of simplex will be truncated can be defined with respect to a \textit{building set}, which is a special kind of a \textit{hypergraph} (see \cite{P09}).
Thus, we get simplices as the limit case in the family, when building set is minimal and where no truncation has been made. As the limit case at the other end, when building set is maximal and where all possible truncations have been made, we have \textit{permutohedra}. There are also other well-known members of this interval, but for needs of this work, beside permutohedron, the most important is an \textit{associahedron} or \textit{Stasheff polytope} (see~\cite{S63}).

The permutoassociahedron arises as a ``hybrid'' of these two nestohedra.
In order to bring the reader closer to our motivation to investigate this compound and have a clearer understanding of its nature and combinatorics, we recall of some combinatorial characteristics of its building elements. For more details on permutohedra and associahedra, we refer to \cite{Z95}, \cite{T06}, \cite{CD06} and \cite{S63}, \cite{BP15}, \cite{S97}, \cite{P09}, respectively.

Combinatorially, the permutohedron is a polytope whose vertices correspond to words obtained by all permutations on $n$ different letters. It can be realised by an $(n-1)$-polytope $\mathbf{P}_n$, whose vertices are obtained by permuting the coordinates of a given generic point in $\mathbf{R}^n$.
Thus, cardinality of the set $\mathcal{V}(\mathbf{P}_n)$ is $n!$. Two vertices are adjacent if and only if their corresponding permutations
%differ by the exchange of two consecutive coordinates,
can be obtained from one another by transposition of two consecutive coordinates, i.e.\ consecutive letters.
Figure~\ref{s:pn} depicts $\mathbf{P}_n$ for $n\in \{2,3,4\}$.
%\vspace{-1mm}
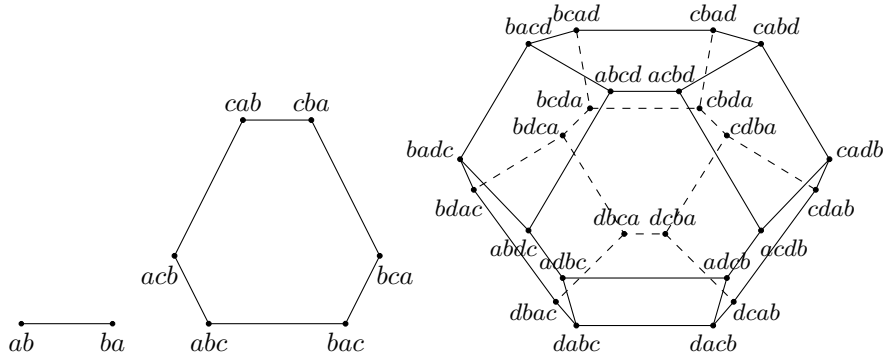
\begin{figure}[h!h!h!]
\begin{center}
\begin{tabular}{ccc}

\begin{tikzpicture}[scale=0.6]
%duz
\draw  (-1,0) node[below]  {$ab$}-- (1,0)node[below]  {$ba$};
\filldraw [black] (-1,0) circle (1.4pt)
                    (1,0) circle (1.4pt);
\end{tikzpicture}
&\hspace{-0.5cm}
\begin{tikzpicture} [scale=0.9]
%sestougao
  \draw
(-1,0) node[below]{$abc$} -- (1,0)node[below]{$bac$}  -- (1.5,1) node[below, xshift=0.2cm]{$bca$} -- (0.5,3) node[above]{$cba$}--(-0.5,3)node[above]{$cab$}--(-1.5,1) node [below, xshift=-0.2cm]{$acb$}-- cycle;
\filldraw [black] (-1,0) circle (1pt)
(1,0) circle (1pt)
(-0.5,3) circle (1pt)
(1.5,1) circle (1pt)
(0.5,3) circle (1pt)
(-1.5,1) circle (1pt)  ;
 \end{tikzpicture}
& \hspace{-0.8 cm}
\begin{tikzpicture} [scale=0.9]
%3D
  \draw  (-1,0) node[below]{\small {$dabc$}} -- (1,0)node[below]{\small {$dacb$}}  -- (1.2,0.7) node[above]{\small {$adcb$}} -- (-1.2,0.7) node[above]{\small {$adbc$}}--cycle;
  \filldraw [black] (-1,0) circle (1pt)
(1,0) circle (1pt)
(1.2,0.7) circle (1pt)
(-1.2,0.7) circle (1pt);

\draw  (1,0) -- (1.3,0.35) node[below, xshift=0.3cm, yshift=0.1cm]{\small {$dcab$}} -- (2.5,2) node[below, xshift=0.2 cm]{\small {$cdab$}}-- (2.7,2.45)node[above, xshift=0.4cm,yshift=-0.1cm]{\small {$cadb$}}-- (1.7,1.4) node[below, xshift=0.3 cm]{\small {$acdb$}} -- (1.2,0.7) ;
 \filldraw [black] (1.3,0.35) circle (1pt)
(2.5,2) circle (1pt)
(2.7,2.45) circle (1pt)
(1.7,1.4) circle (1pt);

 \draw  (-1,0) -- (-1.3,0.35) node[below, xshift=-0.3cm, yshift=0.1cm]{\small {$dbac$}} -- (-2.5,2) node[below, xshift=-0.2 cm]{\small {$bdac$}}-- (-2.7,2.45)node[above, xshift=-0.4cm, yshift=-0.1cm]{\small {$badc$}}-- (-1.7,1.4) node[below, xshift=-0.2 cm]{\small {$abdc$}} -- (-1.2,0.7) ;
 \filldraw [black] (-1.3,0.35) circle (1pt)
(-2.5,2) circle (1pt)
(-2.7,2.45) circle (1pt)
(-1.7,1.4) circle (1pt);

 \draw  (-1.7,1.4) -- (-0.5,3.45) node[above,xshift=0.1cm]{\small {$abcd$}} -- (0.5,3.45) node[above,xshift=-0.1cm]{\small {$acbd$}}-- (1.7,1.4) ;
 \filldraw [black] (-0.5,3.45) circle (1pt)
(0.5,3.45) circle (1pt);

\draw  (-0.5,3.45) -- (-1.7,4.15) node[above]{\small {$bacd$}} -- (-2.7,2.45) ;
 \filldraw [black] (-1.7,4.15) circle (1pt);

\draw  (0.5,3.45) -- (1.7,4.15) node[above,xshift=0.2cm]{\small {$cabd$}} --(1,4.35) node[above]{\small {$cbad$}} --(-1,4.35) node[above]{\small {$bcad$}}--(-1.7,4.15) ;
 \filldraw [black] (1.7,4.15) circle (1pt)
(1,4.35) circle (1pt)
(-1,4.35)circle (1pt);
\draw  (1.7,4.15) -- (2.7,2.45);

\draw[dashed] (-1,4.35) -- (-0.8,3.2)node[above,xshift=-0.4cm,yshift=-0.1cm] {\small{$bcda$}}--(0.8,3.2)node[above,xshift=0.4cm,,yshift=-0.1cm] {\small{$cbda$}}--(1,4.35);
\filldraw [black] (-0.8,3.2) circle (1pt)
(0.8,3.2) circle (1pt);

\draw[dashed] (-0.8,3.2) -- (-1.2,2.8)node[above,xshift=-0.35cm,yshift=-0.1cm] {\small{$bdca$}}--(-2.5,2);
\filldraw [black] (-1.2,2.8) circle (1pt);

\draw[dashed] (-1.2,2.8) -- (-0.3,1.35)node[above,xshift=-0.1cm] {\small{$dbca$}}--(0.3,1.35) node[above,xshift=0.1cm] {\small{$dcba$}}--(1.2,2.8) node[above,xshift=0.35cm,yshift=-0.1cm] {\small{$cdba$}}--(0.8,3.2);
\filldraw [black] (-0.3,1.35) circle (1pt)
(0.3,1.35) circle (1pt)
(1.2,2.8) circle (1pt);

\draw[dashed] (1.2,2.8) --(2.5,2) ;
\draw[dashed] (0.3,1.35) --(1.3,0.35) ;
\draw[dashed] (-0.3,1.35) --(-1.3,0.35) ;
\end{tikzpicture}
\end{tabular}
\end{center}
\caption{Permutohedron $\mathbf{P}_2$, $\mathbf{P}_3$ and $\mathbf{P}_4$} \label{s:pn}
\end{figure}

The associahedron $\mathbf{K}_n$ is an $(n-2)$-polytope whose vertices correspond to complete bracketings in a word of $n$ different letters. Hence, the total number of its vertices is the $(n-1)$th \textit{Catalan} number, i.e.\ cardinality of the set $\mathcal{V}(\mathbf{K}_n)$ is
 \[\frac{1}{n}\binom{2n-2}{n-1}.\]
Two vertices are adjacent if and only if they correspond to a single application of the associativity rule.  The $k$-faces
of the associahedron are in bijection with the set of correct bracketings of an $n$ letters  word with $n-k-1$ pairs of brackets. Two vertices lie in the same $k$-face if and
only if the corresponding complete bracketings could be reduced,
by removing $k$ pairs of brackets, to the same bracketing of the
word of $n$ letters with $n-k-1$ pairs of brackets. Figure~\ref{s:kn}
depicts $\mathbf{K}_n$ for $n \in \{3,4,5\}$.
\begin{figure}[h!h!h!]
\begin{center}
\begin{tabular}{ccc}
\begin{tikzpicture}[scale=0.6]
%duz
\draw  (-1,0) node[below]  {$(ab)c$}-- (1,0)node[below]  {$a(bc)$};
\filldraw [black] (-1,0) circle (1.4pt)
                    (1,0) circle (1.4pt);
\end{tikzpicture}
&\hspace{-0.85cm}
\begin{tikzpicture} [scale=0.9]
%petouao
  \draw
(-1,0) node[below]{$a((bc)d)$} -- (1,0)node[below]{$a(b(cd))$}  -- (1.3,1) node[above, xshift=0.6cm]{$(ab)(cd)$}--(0,3.2)node[above]{$((ab)c)d$}--(-1.3,1) node [above, xshift=-0.6cm]{$(a(bc))d$}-- cycle;
\filldraw [black] (-1,0) circle (1pt)
(1,0) circle (1pt)
(0,3.2) circle (1pt)
(1.3,1) circle (1pt)
(-1.3,1) circle (1pt)  ;
 \end{tikzpicture}
& \hspace{-0.75cm}
\begin{tikzpicture} [scale=1.5]
%3D
  \draw  (-0.75,0) node[below]{\small {$((a(bc))d)e$}} -- (1.3,0)node[below]{\small {$(a((bc)d))e$}}  -- (0.5,2.2) node[above, yshift=-0.1cm]{\small {$(a(b(cd)))e$}} -- (-0.5,2.2) node[above,xshift=-0.85cm,yshift=-0.25cm]{\small {$((ab)(cd))e$}}--(-1.05,0.5) node[below,xshift=-0.2cm]{\small {$(((ab)c)d)e$}}-- cycle;
  \filldraw [black] (-0.75,0) circle (0.75pt)
(1.3,0) circle (0.75pt)
(0.5,2.2) circle (0.75pt)
(-0.5,2.2) circle (0.75pt)
(-1.05,0.5) circle (0.75pt);

 \draw  (1.3,0) -- (1.7,0.3)node[below,xshift=-0.95cm, yshift=0.3cm]{\small {$a(((bc)d)e)$}}  -- (0.9,2.5) node[above, xshift=0.9 cm, yshift=-0.2 cm]{\small {$a((b(cd))e)$}} -- (0.5,2.2);
  \filldraw [black] (1.7,0.3) circle (0.75pt)
(0.9,2.5) circle (0.75pt)
(0.5,2.2) circle (0.75pt);

 \draw  (0.9,2.5) -- (0.7,2.7)node[above, xshift=0.5cm]{\small {$a(b((cd)e))$}}  -- (0.2,2.7) node[above,xshift=-0.5cm]{\small {$(ab)((cd)e)$}} -- (-0.5,2.2) ;
  \filldraw [black] (0.7,2.7) circle (0.75pt)
(0.2,2.7) circle (0.75pt);

\draw[dashed] (-0.75,0) -- (0.4,1)node[above,xshift=-0.4cm,yshift=-0.1cm] {\small{$(a(bc))(de)$}}--(1.1,1) node[above,xshift=0.6cm,yshift=-0.1cm] {\small{$a((bc)(de))$}}--(1.7,0.3);
\filldraw [black](0.4,1) circle (0.75pt)
(0.1,1.5) circle (0.75pt)
(1.1,1) circle (0.75pt);

\draw[dashed] (0.4,1) -- (0.1,1.5)node[below,xshift=-0.85cm,yshift=0.2cm] {\small{$((ab)c)(de)$}}--(-1.05,0.5);
\filldraw [black](0.4,1) circle (0.75pt);

\draw[dashed]  (0.1,1.5)--(0.3,2.1) node[below,xshift=-0.5cm] {\small{$(ab)(c(de))$}}--
(0.8,2.1)node[below,xshift=0.5cm] {\small{$a(b(c(de)))$}}--(1.1,1);
\filldraw [black](0.3,2.1) circle (0.75pt)
(0.8,2.1)circle (0.75pt);

\draw[dashed] (0.7,2.7) --(0.8,2.1) ;
\draw[dashed] (0.2,2.7) --(0.3,2.1) ;
\end{tikzpicture}
\end{tabular}
\end{center}
\caption{Associahedron $\mathbf{K}_3$, $\mathbf{K}_4$ and $\mathbf{K}_5$} \label{s:kn}
\end{figure}
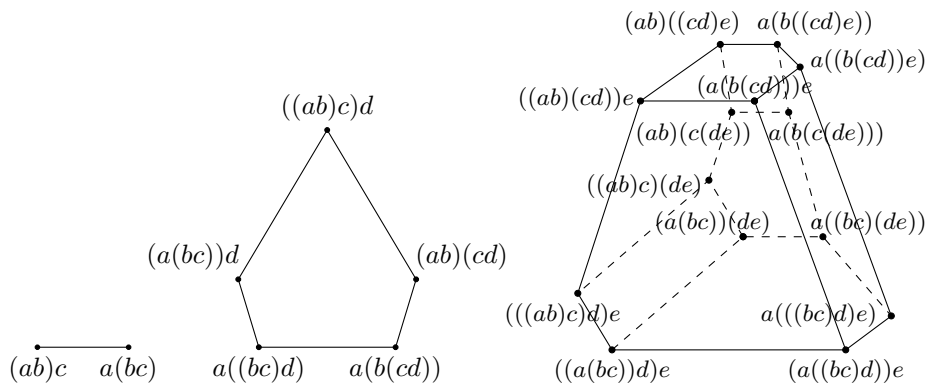

In early 1990s, Kapranov's original motivation for the study of $\mathbf{P}_n$ and $\mathbf{K}_n$ was provided by MacLane's coherence theorem for associativities and commutativities in monoidal categories \cite{ML63}.
He found a ``hybrid-polytope'' that demonstrates interaction between commutativity and associativity, named
permutoassociahedron and denoted by $\mathbf{KP}_n$. It is a polytope whose vertices correspond to all possible complete bracketings of permuted products of $n$ letters. Any $n$ objects in any symmetric (or braided) monoidal category give rise to a diagram of the shape $\mathbf{KP}_n$. He provided its realisation as a combinatorial \emph{CW-complex} and showed that it is an $(n - 1)$-ball. Furthermore, he realised $\mathbf{KP}_3$ and $\mathbf{KP}_4$ as convex polytopes (\cite{K93}). After Kapranov, Reiner and Ziegler gave such a realisation of $\mathbf{KP}_n$ for every $n\geqslant 2$~(\cite{RZ94}).
\begin{figure}[h!h!h!]
\begin{center}
\begin{tabular}{cc}
\begin{tikzpicture} [scale=1.2]
 \draw
(-0.8,0) node[below]{$(ab)c$} -- (0.8,0)node[below]{$(ba)c$}  -- (1.15,0.3) node[below, xshift=0.5cm, yshift=0.2cm]{$b(ac)$} -- (1.3,0.5)node[below, xshift=0.5cm,yshift=0.3cm]{$b(ca)$}--(1.25,0.9)node[below, xshift=0.55cm,yshift=0.2cm]{$(bc)a$}--(0.5,2.7)node[below, xshift=0.55cm,yshift=0.2cm]{$(cb)a$}--(0.2,3)node[above,xshift=0.2cm]{$c(ba)$}--(-0.2,3)node[above,xshift=-0.2cm]{$c(ab)$}--(-0.5,2.7)node[below, xshift=-0.55cm,yshift=0.2cm]{$(ca)b$}--(-1.25,0.9)node[below, xshift=-0.6cm,yshift=0.2cm]{$(ac)b$}--(-1.3,0.5)node[below,xshift=-0.5cm,yshift=0.3cm]{$a(cb)$}-- (-1.15,0.3)node[below,xshift=-0.5cm, yshift=0.2cm]{$a(bc)$}--cycle;

\filldraw [black] (-0.8,0) circle (0.7pt)
(0.8,0) circle (0.7pt)
(1.15,0.3) circle (0.7pt)
(-1.15,0.3) circle (0.7pt)
(1.3,0.5) circle (0.7pt)
(-1.3,0.5) circle (0.7pt)
(1.25,0.9) circle (0.7pt)
(-1.25,0.9) circle (0.7pt)
(0.5,2.7) circle (0.7pt)
(-0.5,2.7) circle (0.7pt)
(0.2,3) circle (0.7pt)
(-0.2,3) circle (0.7pt);
 \end{tikzpicture}
&
\begin{tikzpicture} [scale=1.2]
 \draw
(-0.8,0) node[below]{$(ab)c$} -- (0.8,0)node[below]{$(ac)b$}  -- (1.15,0.3) node[below, xshift=0.5cm, yshift=0.2cm]{$a(cb)$} -- (1.3,0.5)node[below, xshift=0.5cm,yshift=0.3cm]{$c(ab)$}--(1.25,0.9)node[below, xshift=0.55cm,yshift=0.2cm]{$(ca)b$}--(0.5,2.7)node[below, xshift=0.55cm,yshift=0.2cm]{$(cb)a$}--(0.2,3)node[above,xshift=0.2cm]{$c(ba)$}--(-0.2,3)node[above,xshift=-0.2cm]{$b(ca)$}--(-0.5,2.7)node[below, xshift=-0.55cm,yshift=0.2cm]{$(bc)a$}--(-1.25,0.9)node[below, xshift=-0.6cm,yshift=0.2cm]{$(ba)c$}--(-1.3,0.5)node[below,xshift=-0.5cm,yshift=0.3cm]{$b(ac)$}-- (-1.15,0.3)node[below,xshift=-0.5cm, yshift=0.2cm]{$a(bc)$}--cycle;

\filldraw [black] (-0.8,0) circle (0.7pt)
(0.8,0) circle (0.7pt)
(1.15,0.3) circle (0.7pt)
(-1.15,0.3) circle (0.7pt)
(1.3,0.5) circle (0.7pt)
(-1.3,0.5) circle (0.7pt)
(1.25,0.9) circle (0.7pt)
(-1.25,0.9) circle (0.7pt)
(0.5,2.7) circle (0.7pt)
(-0.5,2.7) circle (0.7pt)
(0.2,3) circle (0.7pt)
(-0.2,3) circle (0.7pt);
 \end{tikzpicture}
\end{tabular}
\end{center}
\caption{2-permutoassociahedron $\mathbf{KP}_3$ and $PA_2$} \label{s:kp3pa2}
\end{figure}
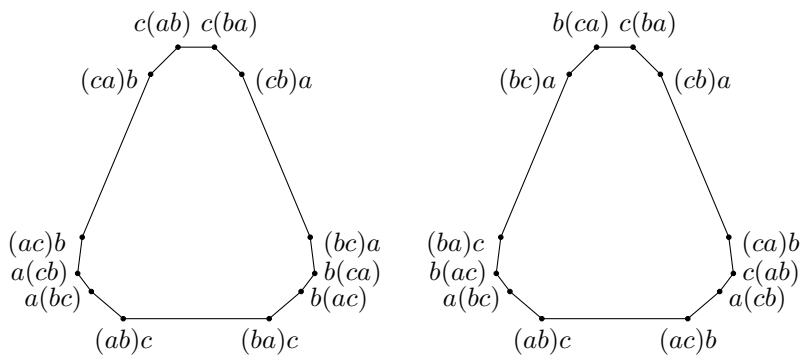

However, for every $n\geqslant 4$, Kapranov's  polytopes are not simple. Even in the case of 3-polytope $\mathbf{KP}_4$, we may notice some vertices that belong to more than three facets. Since these polytopes are hybrids of polytopes that are both simple, it
was natural to search for a family of simple permutoassociahedra.
It was firstly done by Petri\' c in \cite{P14}.
In that paper, he described the \emph{simplicial complex} $C$ obtained by a specific iterative nested construction, whose opposite face semilattice is isomorphic to the face lattice (with $\emptyset$ removed) of the simple $n$-polytope $PA_n$. This polytope is obtained by truncations of $n$-permutohedron such that every vertex expands into an $(n-1)$-associahedron.
Note that the vertices of $PA_n$ can be combinatorially given in the same way as the vertices of $\mathbf{KP}_{n+1}$, but $PA_{n}$ is simple in any dimension. The main difference in approach, which leads to the simplicity of the hybrid-polytope, is a choice of arrows that generate symmetry in a symmetric monoidal category. Namely, there are two types of the edges of $\mathbf{KP}_{n}$ corresponding either to a single reparenthesisation, or to a transposition of two adjacent letters that are grouped together. On the other hand, the edges of $PA_n$ are of the following two types: they also correspond to a single reparenthesisation, or to a transposition of two adjacent letters that are \textit{not} grouped together, i.e.\ to ``the most unexpected'' transposition of neighbours. This essential difference can be recognised even between $\mathbf{KP}_{3}$ and $PA_{2}$, which are both \emph{dodecagons} (see Figure~\ref{s:kp3pa2}).
The 3-dimensional members of these two families of permutoassociahedra are illustrated in Figure~\ref{s:kp4pa3}\footnote{The left illustration is taken from \cite[Section~9.3]{RZ94}, while the right one is made using the graphical algorithm-editor \textit{Grasshopper} (\cite{R18}), a plug-in for Rhinoceros 3D modelling package (\cite{N18}).}. There is a nonsimple vertex of $\mathbf{KP}_{4}$ that corresponds to the word $(bc)(ad)$, which is connected by the edges with the vertices that correspond to the words $(bc)(da)$, $((bc)a)d$, $b(c(ad))$ and $(cb)(ad)$. The vertex of $\mathbf{KP}_{3}$ that correspond to the same word is adjacent just to the three vertices that correspond to the words $((bc)a)d$, $b(c(ad))$ and $(ba)(cd)$.

\begin{figure}[h!h!h!]
\begin{center}
\includegraphics[width=0.77\textwidth]{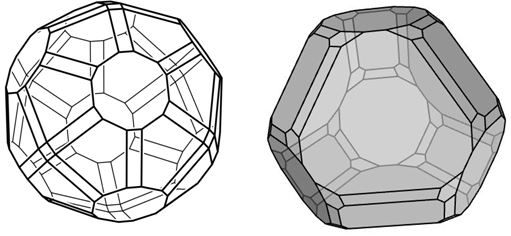}
\caption{3-permutoassociahedron $\mathbf{KP}_4$ and $PA_3$ } \label{s:kp4pa3}
\end{center}
\end{figure}

Based on \cite{P14}, the family of simple permutoassociahedra was further investigated by Curien, Ivanovi\' c and Obradovi\' c (\cite[Section~5.2]{CIO17})
and also by  Barali\' c, Ivanovi\' c and Petri\' c, who gave another explicit realisation with systems of inequalities representing halfspaces in $\mathbf{R}^{n+1}$. This realisation is denoted by $\mathbf{PA}_{n}$ in \cite{BIP17}.
In Section~2, we briefly present the simplicial complex $C$ (the face lattice of a simple permutoassociahedron given combinatorially) and its geometrical realisation $\mathbf{PA}_{n}$.

Since geometrical realisation of this family serves as a topological proof of Mac Lane's coherence, and since it is a generalisation of nestohedra defined by Postnikov as a Minkowski sum of standard simplices (\cite{PRW08}), it is natural to search for an alternative realisation of the simplicial complex $C$, which also uses Minkowski sum as a constructive tool.

Minkowski-decomposability of every simple polytope was confirmed by 
\linebreak Gr\"unbaum more than fifty years ago in \cite[Chapter~15.1, p.\ 321]{G67}, and therefore, decomposability of $\mathbf{PA}_n$ is guaranteed. However, we are interested in finding very specific decomposition (see Definition~\ref{d_Mink_realizacija} below) of its normal equivalent.
Besides Postnikov's representation of nestohedra, there is quite known family of \emph{zonohedra} (also called \emph{zonotopes}, \cite[Section~7.3]{Z95}), defined as Minkowski sum of line segments. There is no other specified representation of any significant family of polytopes, which uses Minkowski sums.
Therefore, the main goal of the papar is to define an $n$-dimensional Minkowski-realisation of the simplicial complex $C$ for every $n\geqslant 2$, according to Definition~\ref{d_Mink_realizacija}, i.e.\ to find an $n$-polytope in $\mathbf{R}^{n+1}$, denoted by $ PA_{n,1}$, which is combinatorially equivalent to $PA_n$ and obtained by Minkowski sums of particular polytopes. The summands are such that each of them leads to the appropriate facet of the whole sum, i.e.\ to a truncation of the currently obtained partial sum.

Before giving the general result, we investigate 2-dimensional Minkowski-realisation of $C$. Namely, in Section 4, we define a 2-polytope $^M PA_2$, normally equivalent to the polytopes $\mathbf{PA}_2$.
Then, in Section 5, for every $n\geqslant 2$, we specify a family of $n$-polytopes  $PA_{n,c}$ for $c\in(0,1]$ such that each member of the family is an $n$-dimensional Minkowski-realisation of $C$ and each one is normally equivalent to $\mathbf{PA}_n$. In particular, the most significant member of the family is $PA_{n,1}$, obtained for $c=1$, because all its summands are defined as convex hulls of points in $\mathbf{R}^{n+1}$. This is particularly beneficial with respect to a computational aspect.

An additional advantage of the new approach using Minkowski sums, is in constructing an algorithm for realisation of the other families of polytopes that also generalise nestohedra. Namely, this research leads to the clear correlation between Minkowski sums and truncations of permutohedron. This implicitly delivers a general procedure for geometrical Minkowski construction of any hybrid of permutohedron and arbitrary nestohedron (permutohedron-based-nestohedron) such that every summand produces exactly one truncation, i.e.\ yields to the appropriate facet of the resulting Minkowski sum.

Throughout the text cardinality of a set $X$ is denoted by $\lvert X \rvert$, $conv\{v_1,\ldots,v_k\}$ represents the convex hull of points $v_1,\ldots,v_k$, the dual space of a vector space $W$ is denoted by $W^\ast$, the set $\{1,\ldots,k\}$ is denoted by $[k]$, the subset relation is denoted by $\subseteq$, while the \textit{proper} subset relation is denoted by~$\subset$. Also, by \emph{comparability}, we mean the comparability with respect to inclusion.
 
\section{nested sets} \label{s_nested}
In this section, we present some known facts about a family of simplicial complexes and two-folded nested sets that are closely related to the  face lattice of $PA_n$.
Since the main goal of the paper is a new geometrical realisation, we omit the theory of nested set complexes in its full generality. We offer just a part of already established theory that is necessary for our research. The following expositions about complexes of nested sets for simplicial complexes and the definition of $\mathbf{PA}_n$ are inherited from \cite{FK04} and \cite{BIP17}, respectively.

\begin{dfn} \label{d_geometrijska_realizacija}\emph{(cf.\ {\cite[p.\ 7]{BIP17}})}
A polytope $P$ \emph{(geometrically) realises} a simplicial complex $K$, when the semilattice obtained by removing the bottom (the empty set) from the face lattice of $P$ is isomorphic to $(K,\supseteq)$.
\end{dfn}

\begin{dfn} \label{d_komb_ekvivalentni politopi}\emph{(cf.\ {\cite[p.\ 38]{G67}})}
Two polytopes $P$ and $Q$ are \emph{combinatorially equivalent}, when their face lattices are isomorphic and it is denoted by $P\sim Q$.
\end{dfn}

\begin{dfn} \label{d_bilding_skup}\emph{(cf.\ {\cite[Definition~3.1]{BIP17}})}
A collection $\mathcal{B}$ of non-empty subsets of a finite set
$V$ containing all singletons $\{v\}, v\in V$ and satisfying that
for any two sets $S_1, S_2\in \mathcal{B}$ such that $S_1\cap S_2
\neq\emptyset$, their union $S_1\cup S_2$ also belongs to
$\mathcal{B}$, is called a \emph{building set} of
$\mathcal{P}(V)$.
Let $K$ be a simplicial complex and let $V_1,\ldots,V_m$ be the maximal simplices of $K$. A collection $\mathcal{B}$ of some
simplices of $K$ is called a \emph{building set} of $K$, when for every
$i\in [m]$, the collection
$$\mathcal{B}_{V_i}=\mathcal{B} \cap \mathcal{P}(V_i)$$ is a
building set of $\mathcal{P}(V_i)$.
\end{dfn}

For a family of sets $N$, $\{X_1,\ldots,X_m\}\subseteq N$ is an $N$\textit{-antichain}, when $m \geqslant 2$ and $X_1,\ldots,X_m$ are mutually incomparable.

\begin{dfn}\label{d_nested_skup_u_odnosu_na_bilding_skup}\emph{(cf.\ {\cite[Definition~3.2]{BIP17}})}
Let $\mathcal{B}$ be a building set of a simplicial complex $K$.
We say that $N\subseteq \mathcal{B}$ is a \emph{nested set} with respect to $\mathcal{B}$, when the union of every $N$-antichain is
an element of $K-\mathcal{B}$.
\end{dfn}

A subset of a nested set is again a nested set. Hence, the nested sets form a simplicial complex.

Now, we proceed to the construction of the building set that gives rise to a simplicial complex of nested sets, which is associated to the simple permutoassociahedron. For $n\geqslant 1$, let $C_0$ be the
simplicial complex $\mathcal{P}\bigl([n+1])-\{[n+1]\bigr\}$, the family of subsets of $[n+1]$ with at most $n$ elements. The simplicial complex $C_0$ is known as the \emph{boundary complex} $\partial \Delta^n$ of
the abstract $n$-simplex $\Delta^n$.

\begin{rem}\label{r_postnikov_i_nas_nested}
The simplicial complex of all nested sets with respect to the building set $\mathcal{B}$ of $C_0$ is isomorphic to the simplicial complex obtained by the collection of all Postnikov's nested sets with removed maximal element from the building set. For more details, we refer to \cite[Section~3]{P14}.
\end{rem}

As a direct corollary of
%Theorem 7.4 in \cite{P09} and Theorem 1 in \cite{P14},
Proposition~9.10 in \cite{DP10},
we have the following claim.

\begin{prop}\label{p_relaizacijaCo}
For every building set $\mathcal{B}$ of $C_0$, there exists a nestohedron $P$ that realises the simplicial complex $K$ of all nested sets with respect to $\mathcal{B}$.
\end{prop}

Such a nestohedron is introduced at the end of Section 3, where we present a polytope $P_\mathcal{B}$ whose semilattice obtained by removing the bottom from its face lattice is isomorphic to $(K,\supseteq)$. This contravariant isomorphism is obtained in such a way that the maximal nested sets correspond to the vertices of the polytope, while the minimal nested sets, i.e.\ the elements of $\mathcal{B}$, correspond to its facets. In general, we say the following.

\begin{dfn}\label{d_proplabel}
Let $P$ be a polytope that realises a simplicial complex $K$ of all nested sets with respect to the building set $\mathcal{B}$ and let $f$ be the contravariant isomorphism.
A facet $F$ of $P$ is \emph{properly labelled} by the element $B$ of $\mathcal{B}$ when $f(F)=\{B\}$.
%The facets of $P$ are \emph{properly labelled} if every facet $F$ is properly labelled by by the element of $\mathcal{B}$.
%Facets of $P$ are \emph{properly labelled} if they are labelled by the corresponding elements of $\mathcal{B}$,
In other words, two facets of $P$ have a common vertex if and only if there is a nested set containing both their labels.
\end{dfn}

Now, let $\mathcal{B}_0=C_0-\{\emptyset\}$. According to Definition~\ref{d_bilding_skup}, $\mathcal{B}_0$
is a building set of $C_0$.
A set $N\subseteq \mathcal{B}_0$ such that the union of every
$N$-antichain belongs to $C_0-\mathcal{B}_0$ is called
0-\emph{nested}. According to Definition~\ref{d_nested_skup_u_odnosu_na_bilding_skup}, every 0-nested
set is a nested set with respect to $\mathcal{B}_0$.
Since a subset of a 0-nested set is also a 0-nested
set, the family of all 0-nested sets makes a new  simplicial
complex $C_1$. Maximal 0-nested sets are
of the form
\[
\bigl\{ \{i_n,\ldots,i_1\},\ldots,\{i_n,i_{n-1}\},\{i_n\} \bigr\},
\]
where $i_1,\ldots,i_n$ are mutually distinct elements of $[n+1]$.

On the other hand, if we consider graph $\Gamma$ with $[n+1]$ as the set of vertices, the set of all members of $C_0$ that are non-empty and connected in $\Gamma$ make a (graphical) building set of $C_0$. Each of these building sets gives rise to a simplicial complex of nested sets, which can be realised as an $n$-nestohedron---a \emph{graph-associahedron} (\cite{CD06}). For example, $n$-permutohedron and $n$-associahedron, correspond to the \emph{complete graph} on $[n+1]$ and the \textit{path graph} $1-\ldots-(n+1)$, respectively.
By the definition of $\mathcal{B}_0$, the simplicial complex of nested sets corresponding to the complete graph on $[n+1]$ is exactly $C_1$, i.e.\ $n$-permutohedron realises $C_1$.

The maximal 0-nested sets correspond to the vertices of the permutohedron such that above-mentioned maximal 0-nested set is associated with the permutation
\[
i_{n+1} i_1\ldots i_n
\]
of $[n+1]$, where $\{i_{n+1}\}=[n+1]-\{i_1,\ldots,i_n\}$. It is easy to see that there is $(n+1)!$ maximal 0-nested sets.
The minimal nested sets of the form $\{B\}$ for $B \in \mathcal{B}_0$, correspond to the facets of the permutohedron. Therefore, two properly labelled facets have a common vertex if and only if their labels are comparable, according to Definition~\ref{d_proplabel}.

Observe that $\mathcal{B}_0$ was defined in a way that covers the recipe for completely truncated simplex, i.e.\ the permutohedron. One can conclude that the next logical step on the road to the permutoassociahedron is to truncate further in order to stretch the interval. Starting from the permutohedron with the recipe that corresponds to the associahedron, we need a new building set of $C_1$ according to a path graph.
Namely, for a maximal 0-nested set
\[
\bigl\{ \{i_n,\ldots,i_1\},\ldots,\{i_n,i_{n-1}\},\{i_n\} \bigr\},
\]
we observe the path graph with $n$ vertices and
$n-1$ edges
\[
\{i_n,\ldots,i_1\}-\ldots-\{i_n,i_{n-1}\}-\{i_n\}.
\]
A set of vertices of this graph is connected,
when this is the set of vertices of a connected subgraph of this graph.

Now, let $\mathcal{B}_1\subseteq C_1$ be the family of all sets of the
form
\[
\bigl\{
\{i_{k+l},\ldots,i_k,\ldots,i_1\},\ldots,\{i_{k+l},\ldots,i_k,i_{k-1}\},\{i_{k+l},\ldots,i_k\}
\bigr\},
\]
where $1\leqslant k\leqslant k+l\leqslant n$ and $i_1,\ldots,i_{k+l}$ are
mutually distinct elements of $[n+1]$, i.e.\ let $\mathcal{B}_1$ be the
set of all non-empty connected sets of vertices of the path graphs that correspond to all maximal 0-nested sets. By Definition~\ref{d_bilding_skup}, $\mathcal{B}_1$ is indeed a building set of the simplicial complex $C_1$.

A set $N\subseteq \mathcal{B}_1$ is 1-\emph{nested} when the union of every $N$-antichain belongs to $C_1-\mathcal{B}_1$. By Definition~\ref{d_nested_skup_u_odnosu_na_bilding_skup}, every 1-nested set is a nested set with
respect to $\mathcal{B}_1$. Again, one can verify that the family
of all 1-nested sets makes a simplicial complex, which is denoted by $C$. For a polytope $P$ that realises $C$, the maximal 1-nested sets correspond to the vertices of $P$, while the singleton 1-nested sets correspond to its facets.
Hence, from the definition of $\mathcal{B}_1$ and Definition~\ref{d_proplabel}, the next claim follows directly.

\begin{prop}\label{p_labele_odrelaizatoraC}
Let $P$ be a polytope that realises  $C$, whose facets are properly labelled by the elements of $\mathcal{B}_1$. Two facets of $P$ have a common vertex if and only if their labels are comparable or the union of their labels is in $C_1-\mathcal{B}_1$.
\end{prop}

According to \cite{BIP17}, a geometrical realisation of $C$ is given as follows. For
\linebreak$1\leqslant k\leqslant k+l\leqslant n$, let
\[
\kappa(k,l)=\frac{3^{k+l+1}-3^{l+1}}{2}+\frac{3^k-3k}{3^n-n-1}.
\]
For an element 
$
\beta=\bigl\{\{i_{k+l},\ldots,i_k,\ldots,i_1\},\ldots,\{i_{k+l},\ldots,i_k,i_{k-1}\},\{i_{k+l},\ldots,i_k\}
\bigr\}$,
of $\mathcal{B}_1$, let $\pi_\beta$ be the equation
(hyperplane in $\mathbf{R}^{n+1}$)
\[
x_{i_1}+2x_{i_2}+\ldots+k(x_{i_k}+\ldots+x_{i_{k+l}})=\kappa(k,l).
\]
\begin{flushleft}
For $\pi$ being the hyperplane $x_1+\ldots +x_{n+1}=3^{n+1}$ in
$\mathbf{R}^{n+1}$, let
\[
\mathbf{PA}_n=(\bigcap \{{\pi_\beta}^{\geqslant}\mid \beta\in \mathcal{B}_1\})\cap
\pi.
\]
\end{flushleft}

\begin{thm}\label{t_realizacijaZoran}\emph{(cf.\ {\cite[Theorem~5.2]{BIP17}})}
$\mathbf{PA}_n \subseteq \mathbf{R}^{n+1}$ is a simple $n$-polytope that realises $C$.
\end{thm}

As a consequence of the previous theorem, Proposition~\ref{p_labele_odrelaizatoraC} and Lemma~5.5 in \cite{BIP17}, we have the following.
\begin{cor} \label{c_labele_odPA_n}
For every $\beta \in \mathcal{B}_1$, the halfspace ${\pi_\beta}^{\geqslant}$ is facet-defining for $\mathbf{PA}_n$. Moreover, if the facets of $\mathbf{PA}_n$ are properly labelled, then the facet
$
\mathbf{PA}_n \cap {\pi_\beta}^{\geqslant}
$
is labelled by $\beta$.
\end{cor}

\section{minkowski sum, normal cones and fans}\label{s_minkowski}
Before we define our main task related to the last theorem, let us recall some facts about normal cones and fans, and also about Minkowski sum, which is one of the fundamental operation on point sets. %Mostly part of the section is based on \cite[Section 4 and 5]{B08},\cite[Section ~7.1]{Z95} and \cite[Section ~15.1]{G67}.
The collection of all polytopes in $\mathbf{R}^n$ is denoted by $\mathcal{M}_n$ (following \cite{B08}).

\begin{dfn}\label{d_suportfunction}
\emph{(cf.\ {\cite[p.\ 36]{B08}})}
The \emph{supporting function} of $P\in \mathcal{M}_n$ is the function
\[s_P:\mathbf{R}^n\longrightarrow \mathbf{R}:s_P(x)=\max\limits_{y\in P}^{} \langle x,y\rangle.
\]
\end{dfn}

For every face $F$ of a polytope $P$, there is a supporting hyperplane through $F$. The set of outward normals to all such hyperplanes spans a polyhedral cone, the normal cone at $F$ (see Figure~\ref{s:skica_fanovi}). More formal definition follows.

\begin{dfn}\label{d_normalconefan}
\emph{(cf.\ {\cite[p.\ 193]{{Z95}}})}
For a given face $F$ of a $d$-polytope $P\in \mathcal{M}_n $, the \emph{normal cone} to $P$ at $F$ is the  collection of linear functionals $v$ in $(\mathbf{R}^d)^\ast$, whose maximum on $P$ is achieved on all the points in the face $F$, i.e.
\[
N_F(P)=\{v\in (\mathbf{R}^d)^\ast \mid  \langle v,y\rangle = s_P(v), \;  \forall y\in F\}.
\]
1-dimensional normal cones are called \emph{rays}.
The \emph{normal fan} of $P$ is the collection
\[
\mathcal{N}(P)=\{N_F(P) \mid F \emph{\text{ is a non-empty face of }} P\}.
\]
\end{dfn}
The normal fan $\mathcal{N}(P)$ is \emph{complete} for every $P \in \mathcal{M}_n$, which means that the union of all normal cones in $\mathcal{N}(P)$ is $\mathbf{R}^n$. 
As we only consider normal fans, the word ``normal'' will be assumed and omitted for brevity from now on. Also, an arbitrary convex cone in $\mathbf{R^{n}}$ of dimension $d \leqslant n$ is written as $d$-cone.
\begin{rem}\label{r_preskmaksimalnihnijemaksimalan}
The intersection of any two normal cones in $\mathcal{N}(P)$ at two faces $F_1$ and $F_2$ is the common face for each of the cones, which is also the normal cone at the smallest face of $P$ that contains both $F_1$ and $F_2$.  
\end{rem}

\begin{exm}\label{e_fanduzi}
The fan of a single line segment $L$ is the set $\{H,H^{\geqslant},H^{\leqslant}\}$, where $H$ is a hyperplane normal to $L$.
\end{exm}

\begin{dfn} \label{d_norm_ekvivalentni politopi}
\emph{(cf.\ {\cite[p.\ 193]{{Z95}}})}
Two polytopes $P,Q \in \mathcal{M}_n$ are called \emph{normally equivalent} when they have the same fan:
\[
P\simeq Q \Leftrightarrow \mathcal{N}(P) = \mathcal{N}(Q).
\]
\end{dfn}
\begin{flushleft}
In literature, normally equivalent polytopes are also called ``analogous'', ``strongly
isomorphic'' or ``related''. The term ``normally equivalent'' is used in \cite{G67} and \cite{Z95}. An example of two normally equivalent polytopes is given in Figure~\ref{s:javaview}.
\end{flushleft}

One can verify that $P\simeq Q \Rightarrow P\sim Q$, but the other direction does not hold. If $Q$ can be obtained from $P$ by parallel translations of the facets, then the outward normals to the corresponding facets of $P$ and $Q$ have the same directions, and then, the rays in $\mathcal{N}(Q)$ and $\mathcal{N}(P)$ coincide. Therefore, the next proposition holds.

\begin{prop} \label{p_paralelni_feseti}
Two combinatorially equivalent polytopes are normally equivalent if and only if their corresponding facets are parallel.
\end{prop}

\begin{rem}\label{r_correspondingfacettruncation}
If $P$ is a polytope in $\mathcal{M}_n$ defined as the intersection of the following $m$ facet-defining halfspaces 
\[
\langle a_i,x\rangle\geqslant b_i,  \;   \; 1\leqslant i\leqslant m,
\]
then a truncation of $P$ in its proper face $F$, \emph{tr}$_FP=P\cap \pi^{\geqslant}$, is the intersection of the following $m+1$ facet-defining halfspaces 
\[
\langle a_i,x\rangle\geqslant b_i,  \;   \; 0\leqslant i\leqslant m,
\]
where $\langle a_0,x\rangle\geqslant b_0$ defines the halfspace $\pi^{\geqslant}$. 

The previous proposition implies the following. For every polytope $Q$ which is normally equivalent to $\emph{tr}_FP$, there exists $c \in \mathbf{R}^{m+1}$ with entries $c_i$ such that $Q$ is the intersection of the following $m+1$ facet-defining halfspaces 
\[
\langle a_i,x\rangle\geqslant c_i,  \;   \; 0\leqslant i\leqslant m.
\]
Hence, if $f$ is a facet of $Q$ lying in the hyperplane $\langle a_0,x\rangle= c_0$ parallel to $\pi$, then there is a bijection $$\mu:\mathcal{F}(Q)-\{f\}\rightarrow \mathcal{F}(P)$$ mapping facets to parallel facets.
We say that facets of polytopes $P$ and $Q$ correspond to each other when they correspond according to $\mu$.
Also, $f$ is called the new appeared facet of $Q$.
\end{rem}

\begin{lem}\label{l_konusiparalelnetrunkacije}
Let \emph{tr}$_FP=P \cap \pi^{\geqslant}$ be a truncation of a given polytope $P \in \mathcal{M}_n$ in its face $F$. For a vertex $u\in F$, let $\{w_i\mid i \in [k]\}$ be the set of vertices of $P$ adjacent to $u$ but not contained in $F$. Also, for every $i \in [k]$, let $E_i=\overline{uw_i}$ and $v_i= E_i\cap \pi$. 
%The following claims hold. 
%\begin{enumerate}
%\item [\emph{(i)}] For every $i \in [k]$, $N_{u}(P)\cap N_{w_i}(P)$ is a face of $N_{v_i}(\emph{tr}_FP)$. 
%\item [\emph{(ii)}] 
The union of all normal cones in $\mathcal{N}(\emph{tr}_FP)$ at the vertices contained in $\pi$ is equal to the union of all normal cones in $\mathcal{N}(P)$ at the vertices contained in $F$. Moreover, if the truncation is parallel, then $$ N_{u}(P)=\bigcup\limits_{i \in [k]}N_{v_i}(\emph{tr}_FP).$$
%\end{enumerate}  
\end{lem}
\begin{proof}  
Without lose of generality, suppose that $P$ is full dimensional. Let $a_0$ be an outward normal to the truncation hyperplane $\pi$.
%\noindent (i): According to Remark~\ref{r_preskmaksimalnihnijemaksimalan}, the intersection $N_{u}(P) \cap N_{w_i}(P)$ is the $(n-1)$-dimensional normal cone $N_{E_i}(P)$ spanned by the outward normals to the facets that contain $E_i$. At the other side, $v_i$ belongs to the truncation hyperplane and the hyperplanes which define all mentioned facets. Thus, $N_{v_i}(\text{tr}_FP)$ is spanned by their outward normals and by the ray $a$, which implies the claim.
%\noindent (ii): 
By the definition of truncation, $N_v(\text{tr}_FP)=N_v(P)$ for every vertex $v$ which is common for both polytopes. Hence, the first part of the claim follows directly from the fact that both fans are complete. Now, let $i$ be an arbitrary element of $[k]$. For every spanning ray $a$ of $N_{v_i}(\text{tr}_FP)$ such that $a \neq a_0$, there is a facet of $P$ which contains $E_i$ and whose an outward normal is $a$, and therefore, $a$ is contained in the cone $N_{u}(P)$. Since the truncation is parallel, there is a hyperplane parallel to $\pi$ which defines $F$, i.e.\ the functional $a_0$ attains the maximum value at $F$ over all points in $P$. It implies that $a_0$ is contained in the normal cone $N_F(P)$. According to Remark~\ref{r_preskmaksimalnihnijemaksimalan}, $N_F(P)$ is common face for all normal cones in $P$ at the vertices of $F$. Therefore, $a_0$ is contained in each of them. In particular, $a_0 \in N_{u}(P)$. We conclude that every spanning ray of $N_{v_i}(\text{tr}_FP)$ is contained in $N_{u}(P)$, which implies $N_{v_i}(\text{tr}_FP)\subseteq N_{u}(P)$. 
\end{proof}

\begin{dfn}\label{d_minkowski}
\emph{(cf.\ {\cite[Definition~1.1.]{B08}})}
Let $A,B\subseteq \mathbf{R}^n$. The \emph{Minkowski sum} of $A$ and $B$ is the set
\[
 A+B=\{x\in \mathbf{R}^n \mid x=x_1+x_2, \;  x_1\in A, \;  x_2\in B\}.
\]
We call $A$ and $B$ the \emph{summands} of $A+B$.
\end{dfn}

%\begin{rem}\label{r_svojstvakonveksnegeometrije}
%Note that the following holds for $A,B,C\subseteq \mathbf{R}^n$. 
%\begin{enumerate}
%\item $conv(A \cup B)=conv(conv(A)\cup conv(B))$;
%\item $(A \cup B)+C=(A+C)\cup (B+C)$.
%\end{enumerate}
%\end{rem}

The Minkowski sum of two polytopes is again a polytope, thus we can use this operation as a classical geometrical constructive tool, which allows us to produce new polytopes from known ones. Moreover, this operation establishes an abelian monoid structure on $\mathcal{M}_n$, where neutral element is the point $0=(0,\ldots,0)\in \mathbf{R}^n$.
Note that $\mathcal{M}_n$ has the structure of $\mathbf{R}$-module, i.e.\ for given $\lambda \in \mathbf{R}$ and $P \in \mathcal{M}_n$
\[
\lambda P =\{\lambda x \in \mathbf{R}^n \mid x\in P\}.
\]

\begin{rem} \label{r_lambdaP}
Scaling a polytope does not change its fan, i.e.\ for every $\lambda > 0$ and $P\in \mathcal{M}_n$, $\lambda P \simeq P$ holds.
\end{rem}

\begin{rem}\label{r_trivialMInk}
For every $0\leqslant \lambda \leqslant 1$ and $P \in \mathcal{M}_n$, $\lambda P$ is  trivially a summand of $P$ for
$$P=\lambda P+(1-\lambda)P.$$
\end{rem}

Through the paper, wherever we are talking about addition of polytopes, we refer to Minkowski sum.

\begin{dfn}\label{d_summand_produces_a_facet}
%Let $P_1$ and $P_2$ be two polytopes whose sum is a polytope $P$.
A polytope $P_2$ is a \emph{truncator summand} for a polytope $P_1$, when there is a truncation $\emph{tr}_FP_1$ of $P_1$ in its proper face $F$ such that
$$P_1+P_2 \simeq \emph{tr}_FP_1.$$
%their sum is a polytope normally equivalent to a polytope \emph{tr}$_FQ$ for some face $F$ of $Q$.
\end{dfn}

\begin{dfn}\label{d_truncator_set}
An indexed set of polytopes $\{P_i\}_ {i\in [m]}$ is a \emph{truncator set of summands} for a polytope $S_0$, when for every $i\in [m]$, $P_i$ is a truncator summand for a polytope $S_{i-1}$, where $S_i=S_{i-1}+P_i$, $i\in[m]$.
\end{dfn}

Now, let $e_i$, $i \in [n+1]$, be the endpoints of the standard basis vectors in $\mathbf{R}^{n+1}$
and let $$\Delta_I = conv\{e_i \mid i \in I\}$$ be the standard $(\lvert I \rvert -1)$-simplex for any given set $I\subseteq [n+1]$.

\begin{dfn}\label{d_Mink_realizacija}
Let $K$ be the simplicial complex of all nested sets with respect to the building set $\mathcal{B}$ and let $\{\mathcal{A}_1,\mathcal{A}_2\}$ be a partition of $\mathcal{B}$ such that the block $\mathcal{A}_1$ is the collection of all singleton elements of $\mathcal{B}$.
An n-polytope $P$ is an $n$-dimensional \emph{Minkowski-realisation} of $K$ when the following conditions are satisfied:
 \begin{enumerate}
   \item[\emph{(i)}] $P$ realises $K$;
   \item[\emph{(ii)}] there exists a function $\varphi:\mathcal{B}\longrightarrow \mathcal{M}_{n+1}$ such that

$$P=\Delta_{[n+1]}+\sum\limits_{\beta \in \mathcal{B}} \varphi(\beta);$$
%and every summand indexed by a non-singleton element of $\mathcal{B}$ produces the facet of $P$ properly labelled by that element.

 \item[\emph{(iii)}] for an indexing function $x: [m]\longrightarrow \mathcal{A}_2$ such that $\vert x(i)\rvert \geqslant \lvert x(j)\rvert$ for every $i<j$, the indexed set $\{P_i\}_{i \in [m]}$, where $P_i=\varphi(x(i))$,
is a truncator set of summands for the partial sum
$$\Delta_{[n+1]}+\sum_{\beta\in \mathcal{A}_1} \varphi(\beta).$$

 \end{enumerate}
\end{dfn}

The main question of the paper follows. It is related to the simplicial complex
$C$ defined in the previous section.

\begin{que}\label{q_pitanje}
How to
%find
define a polytope in $\mathbf{R}^{n+1}$, which is an $n$-dimensional Minkowski-realisation of $C$ and which is normally equivalent to $\mathbf{PA}_n$?
\end{que}

In Section 4, we answer the question in the cases $n=2$, while the general answer for every dimension is given in Section 5. Moreover, we define a family of $n$-polytopes with requested properties.

It is well known that every simple polytope except simplex is \emph{decomposable} (\cite[Chapter~15.1, p.\ 321]{G67}), i.e.\ it can be represented as a Minkowski sum in a nontrivial manner such that the representation possess a summand, which is not positively homothetic to the whole sum (see Remark~\ref{r_trivialMInk}). Thus, decomposability of $\mathbf{PA}_n$ is guaranteed, i.e.\ a nontrivial representation of the family of simple permutoassociahedra as a Minkowski sum exists. But, our goal is very specific representation according to Definition~\ref{d_Mink_realizacija}, and we are searching for a polytope, which does not need to be congruent to $\mathbf{PA}_n$. Still, by the additional request of Question~\ref{q_pitanje}, they have to be normally equivalent. In that manner, requesting normal equivalence between polytopes, we stay on the bridge between coincidence and combinatorial equivalence.

%For solving described problem, we also need the following known claims about Minkowski sum.
Let us recall the following.
\begin{prop}\label{prop_minkowski_svojstva1}
\emph{(cf.\ {\cite[Lemma~1.4.]{B08}})}
If $P_1=conv\{v_1,\ldots,v_k\}$ and \\$P_2=conv\{w_1,\ldots,w_l\}$ are polytopes in $\mathcal{M}_n$, then
\[
P_1+P_2=conv\{v_1+w_1,\ldots,v_i+w_j,\ldots,v_k+w_l\}.
\]
\end{prop}

It follows that for every point $A\in \mathbf{R}^n$, $P+\{A\}$ is a translate of the polytope $P$.  Throughout the text, for two given points $P$ and $T_i$, let $\{P_i\}=\{P\}+\{T_i\}$.

\begin{cor}\label{prop_minkowski_svojstva2}
The following holds in $\mathcal{M}_n$:
\begin{enumerate}
  \item[\emph{(i)}] if $P_1 = P_2$, up to translation, then $P+P_1=P+P_2$, up to translation;
  \item[\emph{(ii)}] if $P=P_1+P_2$, then $\dim(P)\geqslant \max\{\dim(P_1),\dim(P_2)\}$.
\end{enumerate}

\end{cor}

Unlike convexity, simplicity is often violated, i.e.\ the sum of simple polytopes often fails to be simple. Although Minkowski sum is a very simple geometrical operation, its result is not often intuitively predictable and obvious, especially in the case of summing a collection of polytopes of various dimensions or polytopes with a lot of vertices.

Our Question~\ref{q_pitanje} is related with development of the following Postnikov's idea implemented in his Minkowski-realisation of the family of nestohedra (see \cite{P09}).
Let $\mathcal{B}$ be a connected building set of the set $[n+1]$ such that $[n+1]\in \mathcal{B} $. For any set $B \in \mathcal{B}$, we consider the $(\lvert B \rvert -1)$-simplex $\Delta_B$, and the sum
\[
P_\mathcal{B} = \sum_{B\in \mathcal{B}} \Delta_B.
\]
It is shown that this sum is a simple $n$-polytope, which can be obtained by successive parallel truncations of an $n$-simplex, and vice versa, for a nestohedron $P$ and the corresponding building set $\mathcal{B}$, we have $P\sim P_\mathcal{B}$ (see \cite[Theorem~7.4.]{P09}).
%The existance of this realisations could be also taken for Proposition~ \ref{p_relaizacijaCo}.

Note that the following partial sum
$$\Delta_{[n+1]}+\sum_{\substack{B\in \mathcal{B} \\ \lvert B \rvert =1}} \Delta_B$$ is a translate of the $n$-simplex $\Delta_{[n+1]}$ by the point $(1,\ldots,1)\in \mathbf{R}^{n+1}$.
%ova recenica je dodata
For every totally ordered indexing set $I$ of the set of all non-singleton elements of $\mathcal{B}$, such that $\lvert B_i\rvert\geqslant \lvert B_j\rvert$ for $i<j$,
the set $\{\Delta_{B_i}\}_{i \in I}$ is a truncator set of summands for the translated simplex.
%Each of the remaining summands $\Delta_B$ produces an expansion of the simplex and an appearance of a new facet properly labelled by the non-singleton element $B$.
Therefore, according to Definition~\ref{d_Mink_realizacija}, this is indeed an $n$-dimensional Minkowski-realisation of the simplicial complex of all nested sets with respect to $\mathcal{B}-\{[n+1]\}$ (see Remark~\ref{r_postnikov_i_nas_nested}).

\section{the $2$-permutoassociahedron as a minkowski sum} \label{s_dvaD}
In this section we answer Question ~\ref{q_pitanje} in the case $n=2$. By Theorem~\ref{t_realizacijaZoran}, the dodecagon $\mathbf{PA}_{2}$ realises $C$ (see Figure~\ref{s:kp3pa2}). Thus, at the very beginning of this section, we could deliver 12 polytopes  whose sum with $\Delta_{[3]}$ is a dodecagon normally equivalent to $\mathbf{PA}_{2}$ and show that all conditions of Definition~\ref{d_Mink_realizacija} are satisfied.
Instead, we choose another approach, which leads us to the general criteria for finding these summands. As we shall see later in higher dimensions, the most of required summands are neither simplices, nor their sums. Moreover, they need not be even simple polytopes.

According to Section 2, we start with the triangle, i.e.\ the simplicial complex
\[
C_0=\bigl \{ \emptyset,\{1\},\{2\}, \{3\},\{1,2\},\{1,3\},\{2,3\}  \bigr \}
\]
and its building set
\[
\mathcal{B}_0=\bigl \{\{1\},\{2\}, \{3\},\{1,2\},\{1,3\},\{2,3\}  \bigr \},
\]
which leads us to the simplicial complex $C_1$ realised by 2-permutohedron, i.e.\ hexagon.
There are the following 6 maximal 0-nested sets:
\begingroup
\small
\[
\bigl\{\{1,2\},\{1\}\bigr\},\; \bigl\{\{1,2\},\{2\}\bigr\},\;
\bigl\{\{1,3\},\{1\}\bigr\},
\bigl\{\{1,3\},\{3\}\bigr\},\; \bigl\{\{2,3\},\{2\}\bigr\},\;
\bigl\{\{2,3\},\{3\}\bigr\},
\]
\endgroup
and thence,
\[\begin{array}{rll}
\mathcal{B}_1=\Bigl\{ & 
\bigl\{\{1\}\bigr\},\;\bigl\{\{2\}\bigr\},\;\bigl\{\{3\}\bigr\},\;
\bigl\{\{1,2\}\bigr\},\; \bigl\{\{1,3\}\bigr\},\;
\bigl\{\{2,3\}\bigr\} \;  \bigl\{\{1,2\},\{1\}\bigr\}, &
\\ \; & 
\bigl\{\{1,2\},\{2\}\bigr\},\;
\bigl\{\{1,3\},\{1\}\bigr\}, \;
\bigl\{\{1,3\},\{3\}\bigr\},\; \bigl\{\{2,3\},\{2\}\bigr\},\;
\bigl\{\{2,3\},\{3\}\bigr\}& \Bigr\},
\end{array}\]
i.e.\ $\mathcal{B}_1=C_1-\{\emptyset\}$.
According to the elements of the building set, we have the following set of 12 halfspaces:
\[\left\{\begin{array}{lllll}
\text{ } & \text{ }& x_{i_2}& \geqslant &  3
\\
x_{i_1} & + & x_{i_2}& \geqslant &  9
\\
x_{i_1} & + & 2x_{i_2}& \geqslant &  12.5 ,
\end{array}\right.\]
where $i_1$ and $i_2$ are distinct elements of the set $[3]$.
The simplicial complex $C$ is realised by the polytope $\mathbf{PA}_2$ defined as the intersection of the previous set of facet-defining halfspaces and the
hyperplane $x_1+x_2+x_3=27$.

It can be verified that $\mathbf{PA}_2$ is really a dodecagon in $\mathbf{R}^3$. Very efficient tool for such a verification is \texttt{polymake}, an open source software for researches in polyhedral geometry. This computational programme offers a lots of systems, around which one could deal with polytopes in different ways. In particular, there is a possibility to define a polytope as a Minkowski sum of already known ones. For representing $\mathbf{PA}_2$ (see Figure~\ref{s:javaview} left), it is enough to use convex hull codes cdd \cite{F05} and  \texttt{polymake}'s standard tool for interactive visualisation called JavaView \cite{PHPR99}. We extensively use \texttt{polymake} for all verifications that appear in this section.

\begin{figure}[h!h!h!]
\begin{center}
\includegraphics[width=0.8\textwidth]{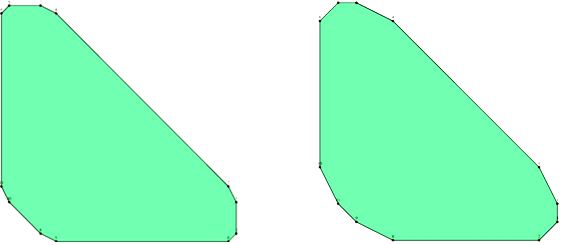}
\caption{JavaView visualisation of $\mathbf{PA}_2$ and $^M PA_2$} \label{s:javaview}
\end{center}
\end{figure}

By Corollary~\ref{c_labele_odPA_n}, all the edges of $\mathbf{PA}_2$ are properly labelled by the elements of $\mathcal{B}_1$ such that the edge labelled by $\beta$ is contained in $\pi_{\beta}$. Also, by Proposition~\ref{p_labele_odrelaizatoraC},
two edges have a common vertex if and only if their labels are comparable.

According to Question~\ref{q_pitanje} and Definition~\ref{d_Mink_realizacija}, our task is to establish a function $\varphi:\mathcal{B}_1\longrightarrow\mathcal{M}_3$ such that if
$$ ^M PA_2=\Delta_{[3]}+\sum_{\beta\in \mathcal{B}_1} \varphi(\beta),$$
then $ ^M PA_2$ is also a dodecagon satisfying Definition~\ref{d_Mink_realizacija}(iii). These two dodecagons have to be normally equivalent. If their edges are properly labelled by the elements of $\mathcal{B}_1$, Proposition~\ref{p_paralelni_feseti} implies that the equilabelled edges have to be parallel.

Let the image of every singleton $\beta\in\mathcal{B}_1$ be the corresponding simplex, i.e.\
\[
\varphi(\beta)= \Delta_{\cup \beta}.
\]
From the end of the previous section, we have that
the partial sum
$$S=\Delta_{[3]}+\sum_{\substack{\beta\in \mathcal{B}_1\\ \lvert \beta \rvert =1}} \varphi(\beta)=\Delta_{[3]}+\Delta_{\{1\}}+\Delta_{\{2\}}+\Delta_{\{3\}}+\Delta_{\{1,2\}}+\Delta_{\{1,3\}}+\Delta_{\{2,3\}}$$
is a completely truncated triangle in $\mathbf{R}^3$, which is a 2-dimensional Minkowski-realisation of the simplicial complex $C_1$. It is a hexagon with three pairs of parallel sides whose edges can be properly labelled by the corresponding $B\subset [3]$. Let us label these edges by $\{B\}$, i.e.\ by the corresponding singleton elements of $\mathcal{B}_1$. Note that they are parallel to the same labelled edges of $\mathbf{PA}_2$.
%Namely, the partial sum of the first four polytopes is a translate of the triangle $e_1e_2e_3$ by the point $(1,1,1)$. We label the sides of the shifted triangle by $\bigl\{\{1\}\bigr\}$, $\bigl\{\{2\}\bigr\}$ and $\bigl\{\{3\}\bigr\}$.
\begin{figure}[h!h!h!]
\begin{center}
\begin{tikzpicture} [baseline, scale=1.3]
    \filldraw (-1.5,1) circle (1pt)
                    (-0.5,1.5) circle (1pt)
                    (0.5,1.5) circle (1pt)
                    (1.5,1) circle (1pt);

 \draw (-2,0)--(-1.5,1)--(-0.5,1.5)--(0.5,1.5)--(1.5,1)-- (2,0);
  \draw [dashed] (-1.5,1)--(-1.25,1.5)--(-0.5,1.5)
        (0.5,1.5)--(1.25,1.5)--(1.5,1) ;
  \draw
   (-2.2,0.4) node[] {\scriptsize $\bigl\{\{1\}\bigr\}$}
   (0,1.7) node {\scriptsize $\bigl\{\{1,2\}\bigr\}$}
  (2.2,0.4) node{\scriptsize $\bigl\{\{2\}\bigr\}$}
  (-1.6,1.3) node {\scriptsize $\bigl\{\{1,2\},\{1\}\bigr\}$}
  (1.6,1.3) node  {\scriptsize $\bigl\{\{1,2\},\{2\}\bigr\}$};
\end{tikzpicture}
\end{center}
\caption{Properly labelled facets} \label{s:skica_labele}
\end{figure}
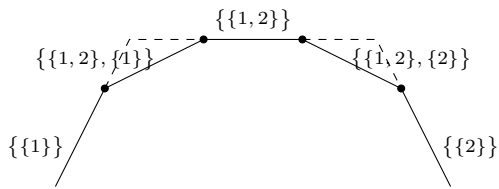

It remains to specify images of six non-singleton elements of $\mathcal{B}_1$, which are of the form $\bigl\{\{i_2,i_1\},\{i_2\}\bigr\}$. According to Definition~\ref{d_Mink_realizacija}(iii), for any order of the summands $\varphi\bigl(\bigl\{\{i_2,i_1\},\{i_2\}\bigr\}\bigr)$, each of them should be a truncator summand for the currently obtained partial sum. 
%In particular,if $\varphi \bigl(\bigl\{\{i_2,i_1\},\{i_2\}\bigr\}\bigr)$ is a polytope $P$, then $S+P$ should be a heptagon obtained by parallel translations of the edges of a polytope tr$_VS$, for some vertex $V$ of $S$.
Let $\{P_i\}_{i \in [6]}$ be an indexed set of all these summands and let us consider all partial sums obtained by adding the elements of this set to the hexagon $S$, step by step. We start with $S_0=S$ and consider every partial sum $S_i=S_{i-1}+P_i$, $i\in[6]$. Notice that $^MPA_2=S_6$.
Since for every $i\in[6]$, there is a truncation of $S_{i-1}$ in some vertex, normally equivalent to $S_{i}$, at $i$th step we can label the edges of $S_i$ by the elements of $\mathcal{B}_1$ in the following way: the corresponding edges of $S_{i}$ and $S_{i-1}$ are equilabelled, while the new appeared edge is labelled by some new label $\beta_i$ (see Remark~\ref{r_correspondingfacettruncation}).
At the end, in order to have all the edges of $S_6$ properly labelled, the following hold for every $i \in[6]$: if $P_i$ corresponds to $\varphi\bigl(\bigl\{\{i_2,i_1\},\{i_2\}\bigr\}\bigr)$, then $S_i \simeq$ tr$_VS_{i-1}$, where $V$ is the common vertex of the edges labelled by $\bigl\{\{i_2,i_1\}\bigr\}$ and $\bigl\{\{i_2\}\bigr\}$, and $\beta_i=\bigl\{\{i_2,i_1\},\{i_2\}\bigr\}$
(see Figure~\ref{s:skica_labele}). Moreover, since the dodecagons are normally equivalent, for every edge of the partial sum $S_i$ there is a parallel equilabelled edge of $\mathbf{PA}_2$.

The proof of the following proposition is quite different from what we discuss here, so it is given later in Section 5.
\begin{prop} \label{p_linijaNE}
If $\varphi$ is a function satisfying the conditions of Definition~\ref{d_Mink_realizacija}, then for every two distinct elements $i_1,i_2\in[3]$, $\varphi\bigl(\bigl\{\{i_2,i_1\},\{i_2\}\bigr\}\bigr)$ is not a line segment.
\end{prop}

From the previous proposition and Corollary~\ref{prop_minkowski_svojstva2}(ii), for every two distinct elements $i_1,i_2\in[3]$, $\varphi\bigl(\bigl\{\{i_2,i_1\},\{i_2\}\bigr\}\bigr)$ is a polygon. Since the order of summands is irrelevant, we start with $\varphi\bigl(\bigl\{\{1,2\},\{1\}\bigr\}\bigr)$ being a triangle $T_1T_2T_3$, where $T_1(a_1,b_1,c_1)$, $T_2(a_2,b_2,c_2)$ and $T_3(a_3,b_3,c_3)$ are points in $\mathbf{R}^3$.
This triangle is a truncator summand for $S$ such that $S+T_1T_2T_3$ is a heptagon normally equivalent to the heptagon obtained from $S$ by truncation in the vertex common for the edges labelled by $\bigl\{\{1,2\}\bigr\}$ and $\bigl\{\{1\}\bigr\}$.
Instead to continue with the whole sum $S$, we consider its partial sum
$$\Delta_{[3]}+\Delta_{\{1\}}
+\Delta_{\{2\}}+\Delta_{\{3\}}+\Delta_{\{1,2\}},$$
which is the trapezoid $ABCD$ given in Figure~\ref{s:skica_12}.
Namely, since the whole sum $S$ is a Minkowki-realisation of $C_1$, its summands indexed by non-singleton sets make a truncator set of the triangle. Hence, we are able to remove some of them such that the sum of the remaining summands has the vertex where the edges labelled by $\{1,2\}$ and $\{1\}$ meet (the vertex $D$).
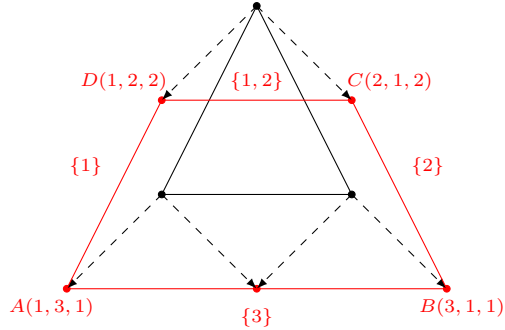
\begin{figure}[h!h!h!]
\begin{center}
\begin{tikzpicture} [baseline, scale=1.25]
  %trougao

  \filldraw [black] (0,1.5) circle (1pt)  %C
                    (-1,-0.5) circle (1pt) %A
                    (1,-0.5) circle (1pt);   %B
    \filldraw [red] (-1,0.5) circle (1pt)  %C1
                    (1,0.5) circle (1pt) %C2
                    (2,-1.5) circle (1pt) %B2
                    (-2,-1.5) circle (1pt) %A1
                 (0,-1.5) circle (1pt);   %A2,B1
   \draw (0,1.5)--(-1,-0.5)--(1,-0.5)--(0,1.5);

   \begin{scope}[>=latex]
   \draw[->,dashed] (0,1.5) -- (-1,0.5);
   \draw[->,dashed] (0,1.5) -- (1,0.5);
   \draw[->,dashed] (-1,-0.5) -- (-2,-1.5);
   \draw[->,dashed] (-1,-0.5) -- (0,-1.5);
   \draw[->,dashed] (1,-0.5) -- (0,-1.5);
   \draw[->,dashed] (1,-0.5) -- (2,-1.5);
   \end{scope}

  \draw [red] (-1,0.5) node[above, xshift=-0.5cm] {\scriptsize $D(1,2,2)$}
  --(1,0.5) node [above, xshift=0.5cm]{\scriptsize $C(2,1,2)$}--(2,-1.5) node[below, xshift=0.2cm] {\scriptsize $B(3,1,1)$}--(-2,-1.5) node[below,xshift=-.2cm] {\scriptsize $A(1,3,1)$}-- (-1,0.5);

   \draw[red]
   (-1.8,-0.2) node[] {\scriptsize $\{1\}$}
   (0,0.7) node {\scriptsize $\{1,2\}$}
   (1.8,-0.2) node {\scriptsize $\{2\}$}
   (0,-1.8) node  {\scriptsize $\{3\}$};
\end{tikzpicture}
\end{center}
\caption{The partial sum $\Delta_{[3]}+\Delta_{\{1\}}+\Delta_{\{2\}}+\Delta_{\{3\}}+\Delta_{\{1,2\}}$} \label{s:skica_12}
\end{figure}

\begin{figure}[h!h!h!]
\begin{center}
\begin{tikzpicture} [baseline, scale=1.4]
  %trapez
  \draw  (1,1) node[above, xshift=0.8cm, yshift=-0.15cm] {\scriptsize $C(2,1,2)$}
 -- (2,-1) node [below, xshift=0.8cm,yshift=0.2cm ]{\scriptsize $B(3,1,1)$}
 -- (-2,-1) node[below, xshift=-0.6cm, yshift=0.2cm] {\scriptsize $A(1,3,1)$}
 -- (-1,1) node[above,xshift=-.6cm, yshift=-0.1cm] {\scriptsize $D(1,2,2)$}
 -- cycle;

   \begin{scope}[>=latex]
    \draw[->,dashed] (2,-1) -- (3,-1.5);   %trojke
   \draw[->,dashed] (-2,-1) -- (-1,-1.5);
   \draw[->,dashed] (-1,1) -- (0,0.5);
    \draw[->,dashed] (1,1) -- (2,0.5);

     \draw[->,dashed] (2,-1) -- (2.5,-0.5); %dvojke
   \draw[->,dashed] (-2,-1) -- (-1.5,-0.5);
    \draw[->,dashed] (1,1) -- (1.5,1.5);
     \draw[->,dashed] (-1,1) -- (-0.5,1.5);

   \draw[->,dashed] (2,-1) -- (1.2,-1.5);  %jedinice
   \draw[->,dashed] (-2,-1) -- (-2.8,-1.5);
    \draw[->,dashed] (1,1) -- (0.2,0.5);
   \draw[->,dashed] (-1,1) -- (-1.8,0.5);

   \end{scope}

   \filldraw[red]
   (-2.8,-1.5) circle (1pt) %A1
  (-1.8,0.5) circle (1pt)  %D1
  (1.5,1.5)circle (1pt)  %C2
   (-0.5,1.5) circle (1pt) %D2
   (3,-1.5) circle (1pt);   %B3

   \draw[red]  (-2.8,-1.5)
 -- (3,-1.5)
 -- (1.5,1.5)
  -- (-0.5,1.5)
 -- (-1.8,0.5)
 -- cycle;

   \filldraw (0.2,0.5) circle (1pt) %C1
   (1.2,-1.5) circle (1pt)  %B1
   (2.5,-0.5) circle (1pt) %B2
   (2,0.5)  circle (1pt) %C3
   (0,0.5)  circle (1pt) %D3
   (-1,-1.5) circle (1pt) %A3
   (-1.5,-0.5) circle (1pt); %A2

  \draw (0.4,0.3) node {\scriptsize $C_1$};
  \draw (-0.2,0.3) node {\scriptsize $D_3$};
  \draw (-1,-1.7) node {\scriptsize $A_3$};
  \draw (2.2,0.4) node {\scriptsize $C_3(2+a_3,1+b_3,2+c_3)$};
  \draw (0.4,-1.7) node {\scriptsize $B_1(3+a_1,1+b_1,1+c_1)$};
  \draw (-3.2,-1.7) node {\scriptsize $A_1(1+a_1,3+b_1,1+c_1)$};
  \draw (-2.9,0.5) node {\scriptsize $D_1(1+a_1,2+b_1,2+c_1)$};
  \draw (2,1.7) node {\scriptsize $C_2(2+a_2,1+b_2,2+c_2)$};
  \draw  (2.6,-0.3) node {\scriptsize $B_2$};
 \draw (-0.6,-0.3) node {\scriptsize $A_2(1+a_2,3+b_2,1+c_2)$};
 \draw (-0.7,1.7) node {\scriptsize $D_2(1+a_2,2+b_2,2+c_2)$};
 \draw (2.8,-1.7) node {\scriptsize $B_3(3+a_3,1+b_3,1+c_3)$};

\end{tikzpicture}
\end{center}
\caption{ $\Delta_{[3]}+\Delta_{\{1\}}+\Delta_{\{2\}}+\Delta_{\{3\}}+\Delta_{\{1,2\}}+T_1T_2T_3$} \label{s:skica_12_1}
\end{figure}
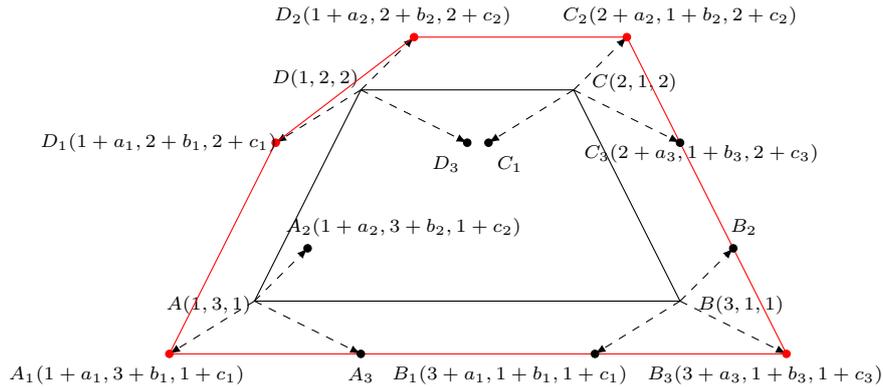

In order to find $T_1T_2T_3$, we focus on an appropriate ``local polytope'', e.g.\ the trapezoid $ABCD$. We assume that $T_1T_2T_3$ is a truncator summand for $ABCD$, which means that
the polytope
$$ABCD+T_1T_2T_3=conv\{A_1,A_2,A_3,B_1,B_2,B_3,C_1,C_2,C_3,D_1,D_2,D_3\},$$
is a pentagon normally equivalent to the polytope obtained from the trapezoid by truncation in the vertex $D$. Also, $-(2,1,0)$ should be an outward normal vector to the new appeared edge.
Let us assume that $D_1D_2$ is that edge such that $D_1$ and $D_2$ are also common for the facets with the outward normal vectors $-(1,0,0)$ and $-(1,1,0)$, respectively (see Figure~\ref{s:skica_12_1}).
We also assume that $A_1,B_3$ and $C_2$ are the vertices of $ABCD+T_1T_2T_3$ such that $A_1$ is common for the edges with the outward normals $-(1,0,0)$ and $-(0,0,1)$, $B_3$ is common for the edges with the outward normal vectors $-(0,1,0)$ and $-(0,0,1)$ and $C_2$ is common for those ones with the outward normal vectors $-(0,1,0)$ and $-(1,1,0)$.
It implies the following system of equations:
\[
\systeme{ a_1+b_1+c_1=a_2+b_2+c_2, a_1+b_1+c_1=a_3+b_3+c_3,2a_1+b_1=2a_2+b_2,c_1=c_3,b_2=b_3. }
\]
The first two follow from the fact that
all translates $A_iB_iC_iD_i$ of the trapezoid $ABCD$, $i \in [3]$, have to lie in the same plane parallel to the plane $x_1+x_2+x_3=5$, in which $ABCD$ lies.

Since $A_2,A_3,B_1,B_2,C_1,C_3 \in conv\{A_1,B_3,C_2,D_1,D_2\}$, we have the following set of inequalities:
\[\begin{array}{lll}
a_1 < a_2 \leqslant a_3, &
b_2 \leqslant b_3 < b_1, &
c_1\leqslant c_3 < c_2 .
%\\[1ex]
\end{array}\]
%\vspace{-5mm}
Solving the system, we get that the points $T_i$ are $$T_1(a_1,b_1,c_1), \;  T_2\bigl(\frac{2a_1+b_1-b_3}{2},b_3,\frac{b_1-b_3}{2}+c_1\bigr),  \;  T_3(a_1+b_1-b_3,b_3,c_1),$$
i.e.
$$T_1(0,b_1-b_3,0), \;  T_2\bigl(\frac{b_1-b_3}{2},0,\frac{b_1-b_3}{2}\bigr),  \;  T_3(b_1-b_3,0,0),$$
up to translation.
It remains to conclude that $T_1T_2T_3$ is any translate of a triangle whose vertices are
$$T_1(0,2\lambda,0), \;  T_2\bigl(\lambda,0,\lambda\bigr),  \;  T_3(2\lambda,0,0),$$
where $\lambda>0$ (see Remark~\ref{r_lambdaP}).
Looking carefully at Figure~\ref{s:skica_12_1}, we can notice that the triangle $ABC$ is also one of them for $\lambda=1$. Let $\varphi \bigl(\bigl\{\{1,2\},\{1\}\bigr\}\bigr)$ be $T_1T_2T_3=conv\{(0,2,0),(1,0,1),(2,0,0)\}$. One can verify that for the vertex $V$ of the hexagon $S$, which is common for the edges labelled by $\bigr \{  \{1,2\}\bigl\}$ and $\bigr \{  \{1\}\bigl\}$, $S+T_1T_2T_3 \simeq$ tr$_VS$ holds, indeed.

Considering an appropriate local polytope, we define images of all non-singleton elements of the building set analogously:
\[
\varphi\bigl(\bigl\{\{i_2,i_1\},\{i_2\}\bigr\}\bigr)=conv\{2e_{i_1},e_{i_2}+e_{i_3},2e_{i_2}\},
\]
where $i_1,i_2$ and $i_3$ are mutually distinct elements of the set $[3]$. Together with already defined images of singleton elements of $\mathcal{B}_1$, we obtained the polytope

$$ ^M PA_2=\Delta_{[3]}+\sum_{\beta\in \mathcal{B}_1} \varphi(\beta).$$

One may verify that $^M PA_2$ is a dodecagon whose vertices are all permutations of the coordinates of the points $(1,5,13)$ and $(2,3,14)$. This dodecagon can also be defined as the intersection of the hyperplane $x_1+x_2+x_3=19$ and the following set of facet-defining halfspaces:  
\[\left\{\begin{array}{lllll}
\text{ } & \text{ }& x_{i_2}& \geqslant &  1
\\
x_{i_1} & + & x_{i_2}& \geqslant &  5
\\
x_{i_1} & + & 2x_{i_2}& \geqslant &  7 ,
\end{array}\right.\]
where $i$ and $j$ are distinct elements of the set $[3]$.
Therefore, two dodecagons are normally equivalent (see Figure~\ref{s:javaview}).
We also verify that Definition~\ref{d_Mink_realizacija}(iii) is satisfied by analysing each partial sum that constitutes $ ^M PA_2$, step by step, for any order of summands.
Finally, according to Definition~\ref{d_Mink_realizacija}, we conclude that $^M PA_2$ is a 2-dimensional Minkowski-realisation of the simplicial complex $C$.

\section{the $n$-permutoassociahedron as a minkowski sum}\label{s_triD}
In the previous section, we gave Minkowski-realisations for 2-permutoassocia-\linebreak hedron handling only with equations of hyperplanes, which define facets of the resulting polytope. We started from local polytopes that were chosen to define particular summands. All verifications were done manually or with a help of \texttt{polymake}. It was done with intention to postpone some definitions and claims about relation between Minkowski sum and fans refinement. However, these matters are necessary for Minkowski-realisation of $n$-permutoassociahedra.

\begin{prop}\label{c_fanovi_sabiraka}
\emph{(cf.\ {\cite[Proposition~7.12.\ and the definition at p.195]{Z95}})}
The fan of the Minkowski sum of two polytopes is the common refinement of their individual fans, i.e.\
\[
\mathcal{N}(P_1+P_2)=\{ N_1 \cap N_2 \mid N_1 \in \mathcal{N}(P_1), N_2\in\mathcal{N}(P_2)\}.
\]
\end{prop}
\begin{proof} [Proof of Proposition~\ref{p_linijaNE}]
Let us suppose that such a line segment $L \subset\mathbf{R}^3$ exists for one pair of distinct elements $i_1,i_2 \in [3]$.
By Proposition~\ref{c_fanovi_sabiraka} and Example~\ref{e_fanduzi}, the fan of the partial sum $S+L$ is the common refinement of the set $\{H,H^{\geqslant},H^{\leqslant}\}$, where $H$ is the plane normal to $L$. Since $S$ is a hexagon with three pairs of parallel sides, its fan is the set consisted of three planes with a common line and six dihedra determined by them. It is straightforward that every refinement of such a fan, which also refines the set $\{H,H^{\geqslant},H^{\leqslant}\}$, always leads either to the same fan or to the fan of an octagon with four pairs of parallel sides. Hence, $S+L$
is not a heptagon, i.e.\ $L$ is not a truncator summand for $S$, which contradicts Definition~\ref{d_Mink_realizacija}(iii).
\end{proof}

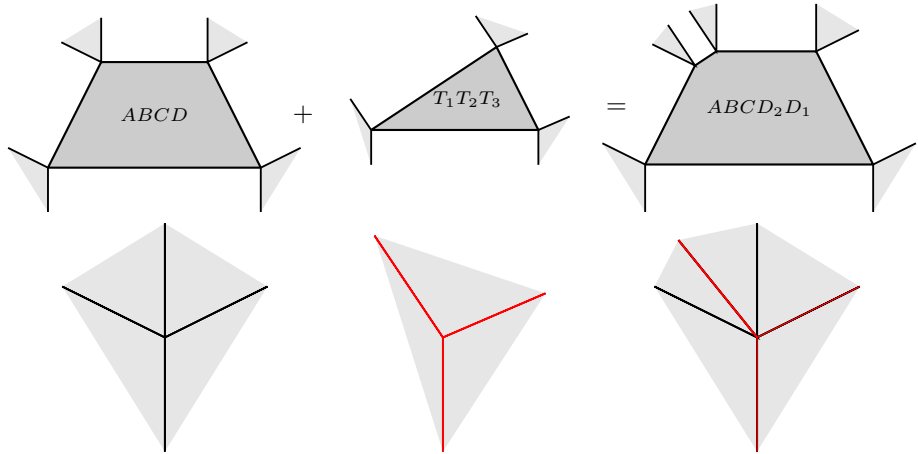
\begin{figure}[h!h!h!]
\begin{center}

\begin{tabular}{cccc}

\begin{tikzpicture}[scale=0.7]
%trapez
\filldraw[fill=gray!40!white, draw=black!40!black,line width=.8pt]
(-1,1) -- (0,-1) -- (-4,-1)-- (-3,1) -- cycle;

\draw  (-2,0) node {\scriptsize $ABCD$};
\draw  (0.8,0) node {$+$};

   \filldraw[fill=gray!20!white, draw=black!40!black,line width=.7pt]
   (-0.25,1.375) -- (-1,1) -- (-1,1.84)     %C
   (-4.75,-0.625) -- (-4,-1) -- (-4,-1.84) %A
   (0,-1.84) -- (0,-1) -- (0.75,-0.625)  %B
   (-3.75,1.375) -- (-3,1) -- (-3,1.84);  %D
\end{tikzpicture}
&
\begin{tikzpicture} [baseline, scale=0.55]
%trougao
  \filldraw[fill=gray!40!white, draw=black!40!black,line width=.8pt]
(0,4) -- (1,2)  -- (-3,2) -- cycle;
   \draw  (-0.7,2.75) node {\scriptsize $T_1T_2T_3$};

  \filldraw[fill=gray!20!white, draw=black!40!black,line width=.7pt]
    (-3.5,2.75) -- (-3,2) -- (-3,1.16)  %T_1
    (-0.5,4.75) -- (0,4) -- (0.75,4.325)  %T_2
   (1,1.16) -- (1,2) -- (1.75,2.325);     %T_3

 \end{tikzpicture}
&
\begin{tikzpicture} [scale=0.75]
%petougao
\filldraw[fill=gray!40!white, draw=black!40!black,line width=.8pt]  (-1,1) %node[above, xshift=0.2cm] {\scriptsize $C$}
 -- (0,-1) %B
 -- (-4,-1) %A
 -- (-3.125,0.75) %D_1
  -- (-2.75,1) %D_2
  -- cycle;
 	\draw  (-2,0) node {\scriptsize $ABCD_2D_1$};
 	\draw  (-4.5,0) node {$=$};

   \filldraw[fill=gray!20!white, draw=black!40!black,line width=.7pt]
   (-0.25,1.375) -- (-1,1) -- (-1,1.84)     %C
   (-4.75,-0.625) -- (-4,-1) -- (-4,-1.84) %A
   (0,-1.84) -- (0,-1) -- (0.75,-0.625)  %B
    (-3.875,1.125) -- (-3.125,0.75) -- (-3.6,1.467) %D1
  (-3.225,1.717) -- (-2.75,1) -- (-2.75,1.84);  %D_2
\end{tikzpicture}

\\
%trapez fan
 \begin{tikzpicture}
 \begin{scope} [scale=1.8]
   \filldraw[fill=gray!20!white, draw=black!40!black,line width=.7pt]
       (-8.75,-2.625) -- (-8,-3) -- (-8,-2.16)  %D
        (-8,-3.84) -- (-8,-3) -- (-7.25,-2.625)  %B
         (-8.75,-2.625) -- (-8,-3) -- (-8,-3.84) %A
       (-7.25,-2.625) -- (-8,-3) -- (-8,-2.16); %C
        \end{scope}

         \end{tikzpicture}
&
 \begin{tikzpicture}
 %trougao fan
 \begin{scope} [scale=1.8]
   \filldraw[fill=gray!20!white, draw=red,line width=.7pt]
       (-0.5,0.75) -- (0,0) -- (0,-0.84)  %T_1
    (-0.5,0.75) -- (0,0) -- (0.75,0.325)  %T_2
   (0,-0.84) -- (0,0) -- (0.75,0.325);     %T_3
        \end{scope}
         \end{tikzpicture}
         &
 \begin{tikzpicture}
 %petougao fan
 \begin{scope} [scale=1.8]
  \filldraw[fill=gray!20!white, draw=black!40!black,line width=.7pt]
       (0.75,0.375) -- (0,0) -- (0,0.84)     %C
   (-0.75,0.375) -- (0,0) -- (0,-0.84) %A
   (0,-0.84) -- (0,0) -- (0.75,0.375)  %B
    (-0.75,0.375) -- (0,0) -- (-0.575,0.717) %D1
     (-0.575,0.717) -- (0,0) -- (0,0.84);  %D_2

    \draw[red, line width=.8pt] (0,0) -- (-0.575,0.717);
    \draw[red, line width=.5pt] (0,0) -- (0.75,0.375);
    \draw[red, line width=.5pt] (0,0) -- (0,-0.84);
        \end{scope}
         \end{tikzpicture}
\end{tabular}
\end{center}
\caption{Normal cones and fans} \label{s:skica_fanovi}
\end{figure}
The previous figure illustrates the common refinement of the individual fans of the trapezoid $ABCD$ and the triangle $T_1T_2T_3$ (see Section 4). The fan of the resulting pentagon is a refinement of the trapezoid's fan by the ray contained in its normal cone at the vertex $D$.
Let us remember that $T_1T_2T_3$ was defined as a translate of the triangle $ABC$, and therefore, its sum with the trapezoid has an edge parallel to $AC$. 

The text below determines a relationship between two polytopes $P_1$ and $P_2$ whose sum is normally equivalent to tr$_F P_1$. In other words, we search for a polytope $P_2$ which is a truncator summand for a given polytope $P_1$. From \cite{P09}, if $P_1$ is a simplex, then $P_2$ is the convex hull of those vertices of $P_1$ that do not belong to $F$. The following proposition shows that the same holds for every simple polytope $P_1$ when $F$ is a vertex. But, if $\dim(F)>0$, then tr$_F P_1 \sim  P_1+P_2$ usually fails. One can find a lot of examples.  
%Namely, first of all, in general, there is no hyperplane which contains all the vertices of $P_2$ adjacent to some vertex in $F$ in $P_1$, which . 

\begin{prop}\label{p_trunkacijatemenaprostog}
Let $P_1\in \mathcal{M}_n$ be an $n$-polytope whose vertex $v$ is contained in exactly $n$ facet. If $P_2=conv(\mathcal{V}(P_1)-\{v\})$, then there is a truncation $\emph{tr}_v P_1$ such that \[P_1+P_2\simeq \emph{tr}_vP_1.\]  
\end{prop}
\begin{proof} 
Let $V=\{v_1,\ldots,v_k\}$ be the set $\mathcal{V}(P_1)-\{v\}$. Since $v$ is contained in exactly $n$ facets, there are exactly $n$ vertices adjacent to $v$ in $P_1$. Let us suppose that $v$ and $v_i$ are adjacent in $P_1$ if and only if $i\in[n]$. 
Then, for every $i\in[n]$, let $w_i$ be the midpoint of the edge $\overline{v_iv}$. Since there exists the hyperplane which contains $w_i$ for every $i\in[n]$, a polytope $conv(\{w_1,\ldots,w_n\} \cup V)$ is a truncation of $P_1$ in $v$, which we denote by tr$_vP_1$. Hence,
\[
\text{tr}_vP_1=conv(\{\dfrac{v+v_1}{2},\ldots,\dfrac{v+v_n}{2}\} \cup V ).
\]

At the other side, from Proposition~\ref{prop_minkowski_svojstva1} and the distributivity low, we have the following equations:
\[\begin{array}{rl}
P_1+P_2= 
& conv(\{v\} \cup V)+conv V=conv\bigl((\{v\}\cup V)+V\bigr)
\\[0.7ex]
= & conv\bigl((\{v\}+V)\cup(V+V)\bigr)=conv\bigl((\{v\}+V)\cup 2V\bigr).
\end{array}\]
The last equation is obtained from the fact that the sum of two different points $v_i$ and $v_j$ is the midpoint of the line segment whose endpoints are $2v_i$ and $2v_j$.
 
By Corollary~\ref{prop_minkowski_svojstva2}(i), we may suppose that $v=0\in \mathbf{R}^n$ without lose of generality. It implies that 
$$P_1+P_2=conv\bigl(V\cup2V\bigr) \; \;\text{and} \; \; \text{tr}_vP_1=conv(\{\dfrac{v_1}{2},\ldots,\dfrac{v_n}{2}\} \cup V).$$
Therefore, $$2\text{tr}_vP_1=conv(\{v_1,\ldots,v_n\} \cup 2V)\subseteq conv(V \cup 2V)=P_1+P_2.$$
For every $j\in[k]-[n]$, we consider the line segment $L_j=\overline{v_jv}$. Since $v$ and $v_j$ are vertices of the polytope $P_1$, this line segment intersects the truncation hyperplane in the point $w$ which belongs to $conv\{w_1,\ldots,w_n\}$. Hence,
$$w=\sum\limits_{i=1}^{n} \alpha_iw_i=\sum\limits_{i=1}^{n} \alpha_i\dfrac{v+v_i}{2}=\sum\limits_{i=1}^{n} \alpha_i\dfrac{v_i}{2},$$ where $\sum\limits_{i=1}^{n} \alpha_i=1$ and $0\leqslant\alpha_i<1$ for every $i\in[n]$. If we suppose that the midpoint of $L_j$ belongs to the line segment $\overline{vw}$, then, since every $w_i$, $i \in [n]$, is the midpoint of the edge adjacent to $v$, we have that $v_j \in 2\overline{wv_j}$. This further implies that $v_j \in conv\{v,v_1,\ldots,v_n\}$, and thus, can not be a vertex of $P_1$. Therefore, the midpoint of $L_j$ belongs to the line segment $\overline{wv_j}$, i.e.\ $\dfrac{v_j}{2}\in conv\{w,v_j\}$. It means that there exist $0<\lambda_1,\lambda_2<1$ such that $\lambda_1+\lambda_2=1$ and \[
\dfrac{v_j}{2}=\lambda_1w+\lambda_2v_j=\lambda_1\sum\limits_{i=1}^{n} \alpha_i\dfrac{v_i}{2}+\lambda_2v_j
,\]
which entails that
\[
v_j=\sum\limits_{i=1}^{n} \lambda_1\alpha_iv_i+\lambda_22v_j
.\]
Since, $0<\lambda_1\alpha_i<1$ for every $i\in[n]$, and $$\lambda_2+\sum\limits_{i=1}^{n} \lambda_1\alpha_i=\lambda_2+\lambda_1\sum\limits_{i=1}^{n} \alpha_i=\lambda_2+\lambda_1=1,$$
we conclude that $v_j \in conv(\{v_1,\ldots,v_n\} \cup 2V)$, which implies $$conv(V \cup 2V)\subseteq conv(\{v_1,\ldots,v_n\} \cup 2V).$$ Hence, $P_1+P_2=2\text{tr}_vP_1$. It remains to apply Remark~\ref{r_lambdaP}.
\end{proof}

\begin{prop}\label{p_dovoljnomakskonuseposmatrati}
For $P,P_1,P_2 \in \mathcal{M}_n$,
$P_1+P_2 \simeq P$ holds if and only if the following two conditions are satisfied. 
\begin{enumerate}
\item[\emph{(i)}] Every maximal normal cone in $\mathcal{N}(P)$ is the intersection of two maximal normal cones in $\mathcal{N}(P_1)$ and $\mathcal{N}(P_2)$.
\item[\emph{(ii)}] If the intersection of two maximal normal cones in $\mathcal{N}(P_1)$ and $\mathcal{N}(P_2)$ is an $n$-cone, then it is a maximal normal cone in $\mathcal{N}(P)$.
\end{enumerate}
\end{prop}
\begin{proof}
%The other direction follows from the fact that the fans of both polytopes are complete. 
Suppose that the sets of maximal normal cones in the fans of two arbitrary polytopes are equal. Each normal cone in one of the fans is a face of some maximal normal cone in that fan. Then, by assumption, it is also a face of same maximal normal cone in the other fan. Hence, by Definition~\ref{d_normalconefan}, that cone is contained in both fans. This, together with Definition~\ref{d_norm_ekvivalentni politopi} and Definition~\ref{d_normalconefan}, implies that two polytopes are normally equivalent if and only if the sets of maximal normal cones in their fans are equal. 
%In particular,
%$$(\star)\quad \{N_v(P_1+P_2)\mid v \in \mathcal{V}(P_1+P_2)\}=\{N_v(P)\mid v \in \mathcal{V}(P)\}.$$ 
It remains to apply Proposition~\ref{c_fanovi_sabiraka}.
\end{proof} 

According to Remark~\ref{r_correspondingfacettruncation}, let $P_1\in \mathcal{M}_n$ be $d$-polytope defined as the intersection of the following $m$ facet-defining halfspaces
\[
\alpha_i^{\geqslant}: \langle a_i,x\rangle\geqslant b_i,  \;   \; 1\leqslant i\leqslant m, \]
and let tr$_FP_1$ be a parallel truncation of $P_1$ in its face $F$ defined as the intersection of the following $m+1$ facet-defining halfspaces
\[\alpha_i^{\geqslant}: \langle a_i,x\rangle\geqslant b_i,  \;   \; 0\leqslant i\leqslant m.\]    

\begin{dfn}\label{d_pideformacija}
Let $P_1$ and \emph{tr}$_FP_1$ be two polytopes defined as above. A polytope $P_2\in \mathcal{M}_n$ is an $F$-\emph{deformation}\footnote{This definition is inspired by {\cite[Definition~15.01]{PRW08},} which defines several types of deformation cones of a given polytope.} of $P_1$ when the following conditions are satisfied:
\begin{enumerate}
%\item [\emph{(i)}] $\dim(P_2)=\dim(P_1)$;
\item [\emph{(i)}] $P_2$ is the intersection of the halfspaces 
$$\pi_i^{\geqslant}: \langle a_i,x\rangle\geqslant c_i, \; 0\leqslant i\leqslant m, \text{ such that } P_2 \cap \pi_0\simeq P_1\cap \alpha_0;$$
%the hyperplane $\pi_0$ defines a $(d-1)$-face of $P_2$ which is normally equivalent to $P_1\cap \alpha_0$; 
%the halfspace $\pi_0^{\geqslant}$ is facet-defining;
\item [\emph{(ii)}]for every $S\subset\{0,\ldots,m\}$ 
\[\bigcap\{\alpha_i\mid i \in S\} \text{ is a vertex of } \emph{tr}_FP_1 \Rightarrow\bigcap\{\pi_i\mid i \in S\} \text{ is a vertex of } P_2.\]
\end{enumerate}
\end{dfn}

\begin{rem}\label{r_deformacijagruboreceno}
The condition \emph{(ii)} together with the first part of condition \emph{(i)} means that $P_2$ can be obtained from tr$_FP_1$ by parallel translations of the facets without, roughly speaking, crossing over the vertices \footnote{``...by moving the vertices such that directions of all edges are preserved (and some edges may accidentally degenerate into a single point).''\cite{P09}}. If $f$ is the facet of \emph{tr}$_FP_1$ contained in the truncation hyperplane, then the second part of the condition \emph{(i)} implies that $d-1\leqslant \dim(P_2)\leqslant d$, and that $\pi_0$ is a supporting hyperplane for $P_2$ defining a $(d-1)$-face normally equivalent to $f$. If $\dim(P_2)=d-1$, that face is $P_2$ itself.
\end{rem}
%\begin{rem}\label{r_normalnoekvivalentandeformaciji}
%If $P_1 \simeq Q_1$, $P_2\simeq Q_2$ and $P_2$ is an $F$-de\-for\-ma\-tion of $P_1$, then $Q_2$ is an $F$-deformation of $Q_1$.  
%\end{rem}
\begin{rem}\label{r_korespondingdeformacijateme}
We say that a vertex $v$ of an $F$-deformation of $P_1$ corresponds to some vertex $u$ of $\emph{tr}_FP_1$ if $v$ corresponds to $u$ according to \emph{Definition~\ref{d_pideformacija}(ii)}. 
\end{rem}
\begin{exm}\label{e_trunkacijajedeformacija}
Every parallel truncation \emph{tr}$_FP$ of an arbitrary polytope $P$ is an $F$-deforma\-tion of $P$.
\end{exm}

\begin{exm}
The triangle $ABC$ is a $D$-deformation of the trapezoid $ABCD$ illustrated in Figure~\ref{s:skica_12}.
Figures~\ref{s:skica_a4}, \ref{s:skica_a2} and \ref{s:skica_a3} depict some 3-nestohedra and their deformations.
\end{exm}

\begin{lem}\label{l_zvezda}
Let $P_1$ and \emph{tr}$_FP_1$ be two polytopes defined as above. If $P_2$ is an $F$-deformation of $P_1$, then the following claims hold.

\begin{enumerate}

\item[\emph{(i)}] For $v$ being a vertex of $\emph{tr}_FP_1$ and $u_2$ being its corresponding vertex of $P_2$, we have that $N_v(\emph{tr}_FP_1)\subseteq N_{u_2}(P_2)$.

\item[\emph{(ii)}] Every maximal normal cone in $\mathcal{N}(P_2)$ is the union of some maximal normal cones in $\mathcal{N}(\emph{tr}_FP_1)$.

\item[\emph{(iii)}] Every maximal normal cone in $\mathcal{N}(P_2)$ contains no more than one maximal normal cone to $\emph{tr}_FP_1$ at some vertex contained in the truncation hyperplane.
\end{enumerate}
\end{lem}
\begin{proof}
Without lose of generality, suppose that $P_1$ is full dimensional.

\noindent (i): Let $v$ be contained in the facets defined by the halfspaces $\{\alpha_i^{\geqslant} \mid i\in S\}$. Then, the set of rays $\{-a_i \mid i \in S\}$ spans $N_v(\text{tr}_FP_1)$, and $u_2$ is the intersection of the hyperplanes $\pi_i$, $i \in S$. Therefore, for every $i \in S$ the functional $-a_i$ attains the maximum value $-c_i$ at $u_2$ over all points in $P_2$, which implies that $-a_i$ is in the cone $N_{u_2}(P_2)$. Since all these rays are spanning rays of $N_v(\text{tr}_FP_1)$, the claim holds. 

\noindent(ii):  By the previous claim, for every maximal normal cone $N\in \mathcal{N}(\text{tr}_FP_1)$ there is a maximal normal cone in $\mathcal{N}(P_2)$ in which $N$ is contained. Since $\mathcal{N}(\text{tr}_FP_1)$ and $\mathcal{N}(P_2)$ are complete, the claim holds.  

\noindent(iii): By Definition~\ref{d_pideformacija}(i), $P_2 \cap \pi_0$ is an $(n-1)$-face of $P_2$, and hance, the ray $-a_0$ is a spanning ray just of those normal cones to $P_2$ that correspond to the vertices of that face. By the claim (ii), each of them contains at least one normal cone to tr$_FP_1$ at some vertex contained in $\alpha_0$.  
Then, since $P_2 \cap \pi_0 \sim \text{tr}_FP_1 \cap \alpha_0$, the claim follows directly from the equation $\lvert\mathcal{V}(P_2 \cap \pi_0)\rvert=\lvert\mathcal{V}(\text{tr}_FP_1 \cap \alpha_0)\rvert$.

\end{proof}

\begin{prop}\label{l_deformacijajesingleprofinjenje}
Let $P_1$ and \emph{tr}$_FP_1$ be two polytopes defined as above. If $P_2$ is an $F$-deformation of $P_1$, then
%\emph{tr}$_FP_2=P_2\cap \pi^{\geqslant}$,
$$P_1+P_2 \simeq \emph{tr}_FP_1.$$
\end{prop}
\begin{proof}
Without lose of generality, we suppose that $P_1$ is full dimensional and show the claim according to Proposition~\ref{p_dovoljnomakskonuseposmatrati}.
%it is enough to show that
%$$(\star)\quad \{N_v(P_1+P_2)\mid v \in \mathcal{V}(P_1+P_2)\}=\{N_v(\text{tr}_FP_1)\mid v \in \mathcal{V}(\text{tr}_FP_1)\}.$$
%Firstly, we show that the set at the right-hand side of the equation $(\star)$ is a subset of the set at its left-hand side. 

Let $N_v$ be a normal cone to $\text{tr}_FP_1$ at a vertex $v$. The goal is to find two maximal normal cones $N_1 \in \mathcal{N}(P_1)$ and $N_2 \in \mathcal{N}(P_2)$ such that $N_v=N_1 \cap N_2$. Let $u_2$ be a vertex of $P_2$ which corresponds to $v$ according to Remark~\ref{r_korespondingdeformacijateme}. Lemma~\ref{l_zvezda}(i) guarantees that $N_v \subseteq N_{u_2}(P_2)$. 
%Suppose that $v$ is the intersection of the hyperplanes $\{\alpha_i \mid i\in S\}$, i.e.\ $N_v$ is spanned by the rays $-a_i$, where  $i\in S$. By Definition~\ref{d_pideformacija}, there is a vertex $v_2$ of $P_2$ which is the intersection of the hyperplanes $\{\pi_i \mid i \in S\}$. 
If $-a_0$ is not a spanning ray of $N_v$, then $v$ is also a vertex of $P_1$, i.e.\ $N_v=N_v(P_1)$. Then, $N_v$ is the intersection of the maximal normal cones $N_v(P_1)$ and $N_{u_2}(P_2)$. Otherwise, i.e.\ if $v$ belongs to the truncation hyperplane $\alpha_0$, then there is an edge $E$ of $P_1$ which has a common vertex with $F$ intersecting $\alpha_0$ in $v$ .  
Let $u_1$ be that vertex, i.e.\ $u_1=E\cap F$.
By Lemma~\ref{l_konusiparalelnetrunkacije}, $N_v\subseteq N_{u_1}(P_1)$, and hence, $N_v \subseteq N_{u_1}(P_1)\cap N_{u_2}(P_2)$. Now, there are two possible cases. If $N_{u_2}(P_2)=N_v$, then $N_v$ is the intersection of $N_{u_1}(P_1)$ and $N_{u_2}(P_2)$. Otherwise, by Lemma~\ref{l_zvezda}(ii) and (iii), $N_{u_2}(P_2)=N_v \cup N$, where $N$ is the union of some maximal normal cones in $\mathcal{N}(\text{tr}_FP_1)$ such that each of them corresponds to some vertex not contained in the truncation hyperplane, i.e.\ to some vertex of $P_1$ not contained in $F$. If we suppose that $N_v \subset N_{u_1}(P_1)\cap N_{u_2}(P_2)$, then there is a maximal normal cone in $\mathcal{N}(P_1)$ which is contained in $N$ and whose intersection with $ N_{u_1}$ is an $n$-cone. This is contradiction since they are maximal normal cones in the same fan (see Remark~\ref{r_preskmaksimalnihnijemaksimalan}), and hence, $N_v=N_{u_1}(P_1)\cap N_{u_2}(P_2)$. We conclude that the first condition of Proposition~\ref{p_dovoljnomakskonuseposmatrati} is satisfied.

%If there would be an interior ray in $N$ is the union of maximal cones in the same $N_{u_1}$ are , we obtain that the intersection of $N_{u_1}$ and $N_{u_2}$ is again $N_v$.  

%Suppose that there is a ray $r$ which is contained in $N_{u_1}(P_1)$ and $N_{u_2}(P_2)$ but not contained in $N_v$. Then, $r$ is contained in $N-N_v$, Since $\mathcal{N}(P_1)$ is complete, $r$ is not in $N_{v_1}(P_1)$. Therefore, $N_v \supseteq N_{u_1}(P_1)\cap N_{u_2}(P_2)$ also holds.

Let $N_{u_1}$ and $N_{u_2}$ be two normal cones to $P_1$ and $P_2$ at a vertex $u_1$ and $u_2$, respectively. The goal is to show that if their intersection is a maximal cone, then it is a maximal cone in $\mathcal{N}(\text{tr}_FP_1)$.
If $u_1\notin F$, then $N_{u_1} =N_{u_1}(\text{tr}_FP_1)$. By Lemma~\ref{l_zvezda}(i), $N_{u_2}$ is the union of some maximal cones in $\mathcal{N}(\text{tr}_FP_1)$, and thus, the intersection of $N_{u_1}$ and $N_{u_2}$ is a maximal cone if and only if $N_{u_1}\subseteq N_{u_2}$. In that case, their intersection is exactly $N_{u_1}$, a maximal cone in $\mathcal{N}(\text{tr}_FP_1)$. 
Now, let $u_1$ be a vertex contained in $F$. By Lemma~\ref{l_zvezda}(iii), we have two possible cases for $N_{u_2}$. If all of the maximal normal cones that are contained in $N_{u_2}$ correspond to vertices of $\text{tr}_FP_1$ not contained in the truncation hyperplane, then all of them are maximal normal cones to $P_1$ at vertices that do not belong to $F$. Therefore, according to Remark~\ref{r_preskmaksimalnihnijemaksimalan}, the intersection of $N_{u_1}$ and the union of such cones is not an $n$-cone. Otherwise, there is exactly one vertex $v$ of tr$_FP_1$ contained in the truncation hyperplane, such that $N_v(\text{tr}_FP_1)\subseteq N_{u_2}$. According to Lemma~\ref{l_konusiparalelnetrunkacije} and Remark~\ref{r_preskmaksimalnihnijemaksimalan}, the intersection of $N_{u_1}$ and $N_{u_2}$ is an $n$-cone if and only if there is an edge of $P_1$ containing both $u_1$ and $v$. When it is a case, their intersection is exactly $N_v(\text{tr}_FP_1)$. We conclude that the second condition of Proposition~\ref{p_dovoljnomakskonuseposmatrati} is also satisfied.
\end{proof}

If $P_1$ is simple polytope with a vertex $v$, then $conv(\mathcal{V}(P_1)-\{v\})$ is a $v$-de\-for\-mation of $P_1$. It means that Proposition~\ref{p_trunkacijatemenaprostog} is just a special case of the previous one. However, the methods used in theirs proofs are essentially different (note that Proposition~\ref{c_fanovi_sabiraka} is not even used in the proof of Proposition~\ref{p_trunkacijatemenaprostog}). 

Now, in order to answer Question~\ref{q_pitanje}, we present a polytope $PA_{n,1}$, and furthermore, a family of $n$-polytopes $PA_{n,c}$, where $c\in (0,1]$.
Let $\{\mathcal{A}_1,\mathcal{A}_2\}$  be a partition of $\mathcal{B}_1$ such that the block $\mathcal{A}_1$ is the collection of all the singletons, i.e.
\[
\mathcal{A}_1=\bigl\{\{\{i_{1+l}, \ldots, i_1\}\} \mid 0\leqslant l \leqslant n-1\bigr\} \text{ and } \mathcal{A}_2=\mathcal{B}_1-\mathcal{A}_1,
\]
where $i_1,\ldots,i_n$ are mutually distinct elements of $[n+1]$.
For the sequel, let 
\[
\beta=\bigl\{\{i_{k+l},\ldots,i_k,\ldots,i_1\},\ldots,\{i_{k+l},\ldots,i_k,i_{k-1}\},\{i_{k+l},\ldots,i_k\}
\bigr\},
\]
be an element of $\mathcal{A}_2$, where $1 < k\leqslant k+l\leqslant n$.
Let $$\beta_{min}=\{i_{k+l},\ldots,i_k\}, \quad \beta_{max}=\{i_{k+l},\ldots,i_k,\ldots,i_1\} \quad \text{and}$$ 
\[
\mathcal{B}_\beta= \bigl\{B \subseteq [n+1]\mid B \in \beta \text{ or } B\subset \beta_{min}
 \text{ or } \beta_{max}\subset B \bigl\} \cup \bigl\{\{v\} \mid v \in [n+1]\bigr\}.
\]

\begin{lem} \label{l_bildingzanest}
The set $\mathcal{B}_\beta-\{[n+1]\} $ is a building set of $\mathcal{P}([n+1])$.
\end{lem}
\begin{proof}
Let $B_1$ and $B_2$ be two distinct elements of $\mathcal{B}_\beta-\{[n+1]\}$ such that $B_1\cap B_2 \neq \emptyset$. Hence, they are not singletons. If they are comparable, then their union belongs to $\mathcal{B}_\beta -\{[n+1]\}$. Otherwise, since $\beta_{min}\subset \beta_{max} $, we have that $B_1,B_2\supset \beta_{max}$  or  $B_1,B_2\subset \beta_{min}$. It follows that $\beta_{max} \subseteq B_1\cap B_2 \subset B_1\cup B_2$ or $\beta_{min} \supseteq B_1\cup B_2$. Hence, $\mathcal{B}_\beta-\{[n+1]\}$ is a building set of $\mathcal{P}([n+1])$ according to Definition~\ref{d_bilding_skup}.
\end{proof}

By the previous lemma and Definition~\ref{d_bilding_skup}, $\mathcal{B}_\beta-\{[n+1]\} $ is a building set of the simplicial complex $C_0$. The family of all nested sets with respect to this building set forms a simplicial complex, which we denote by $C_2$.

\begin{prop} \label{p_nestoedarglavni}
The nestohedron $P_{\mathcal{B}_\beta }$
is an $n$-dimensional Minkowski-realisation of $C_2$.
\end{prop}
\begin{proof}
It follows directly from Lemma~\ref{l_bildingzanest}, Remark~\ref{r_postnikov_i_nas_nested}, Proposition~\ref{p_relaizacijaCo} and Postnikov's Minkowski-realisation of nestohedra given in the end of Section 3.
\end{proof}

Therefore, the facets of $P_{\mathcal{B}_\beta}$ can be properly labelled according to Definition~\ref{d_proplabel}. For an element $A \in \mathcal{B}_\beta-\{[n+1]\}$, let $f_A$ be the facet labelled by $A$.
By the definition of $\mathcal{B}_\beta$, we have that $\beta \subseteq \mathcal{B}_\beta-\{[n+1]\} $, and hence, let
\[
F_\beta= \bigcap\limits_{B \in \beta} f_B .
\]

Since the elements of $\beta$ are mutually comparable, $F_\beta$ is a proper face of  $P_{\mathcal{B}_\beta}$ (see (cf.\ {\cite[Theorem~1.5.14]{BP15}})), and since $\beta$ is not a singleton, $F_\beta$ is not a facet. 

Let $ \mathcal{B}_{\beta \vert A} $ denote $\{ B \in \mathcal{B}_\beta \mid B\subseteq A\}$. 

\begin{prop}
\label{p_burstabernejednakosti}
We have
\[
P_{\mathcal{B}_\beta}=\bigl \{  x \in \mathbf{R}^{n+1} \mid	\sum\limits_{i=1}^{n+1}x_i=\lvert \mathcal{B}_\beta \rvert , \; 	\sum\limits_{i \in A}^{}x_i\geqslant \lvert \mathcal{B}_{\beta \vert A} \rvert \emph{\text{ for every }} A\in \mathcal{B}_\beta \bigr \}.
\]
\\Moreover, every hyperplane $H_A=\bigl \{  x \in \mathbf{R}^{n+1}  \mid	\sum\limits_{i \in A}^{}x_i=\lvert \mathcal{B}_{\beta \vert A} \rvert \; \bigr \}$  with $A \neq [n+1]$ defines the facet $f_A$ of $P_{\mathcal{B}_\beta}$.
\end{prop}
\begin{proof}
It follows directly from Proposition~1.5.11. in \cite{BP15} and Proposition~\ref{p_nestoedarglavni}.
\end{proof}

Now, let $N_\beta$ be a polytope obtained from $P_{\mathcal{B}_\beta}$ by removing the face $F_\beta$, i.e.
\[
N_\beta=conv\bigl( \;	\mathcal{V}(P_{\mathcal{B}_\beta})-\mathcal{V}(F_\beta) \;	\bigr).
\]
Let $\kappa_\beta:\mathbf{R}^{n+1}\rightarrow \mathbf{R}$ be a function such that
\[
\kappa_\beta(x)=\sum\limits_{B \in \beta}^{} \sum\limits_{i \in B} x_i=x_{i_1}+2x_{i_2}+\ldots+k(x_{i_k}+\ldots+x_{i_{k+l}}),
\]
where $x=(x_1,\ldots,x_{n+1})$, and let
$
m_\beta= \min\limits_{v\in \mathcal{V}(P_{\mathcal{B}_\beta})}^{} \kappa_\beta(v)$.

\begin{prop}\label{p_prekominimumaFbeta}
The following holds:
\[
F_\beta=conv\{v\in \mathcal{V}(P_{\mathcal{B}_\beta}) \mid  \kappa_\beta(v)= m_\beta\}.
\]
\end{prop}
\begin{proof}
Let $v=(v_1,\ldots,v_{n+1})$ be a vertex of the nestohedron $\mathcal{P_{\mathcal{B}_\beta}}$.
Since $\beta\subseteq \mathcal{B}_\beta$, from Proposition~\ref{p_burstabernejednakosti}, we have that
\[
\kappa_\beta(v)=\sum\limits_{B \in \beta}^{} \sum\limits_{i \in B} v_i \geqslant \sum\limits_{B \in \beta}^{} \lvert \mathcal{B}_{\beta \vert B} \rvert ,
\]
which implies $m_\beta=\sum\limits_{B \in \beta}^{} \lvert \mathcal{B}_{\beta \vert  B} \rvert $. Therefore, $\kappa_\beta(v)=m_\beta$ if and only if for every $B \in \beta$, the vertex $v$ lies in the hyperplane $H_B$. Since, $H_B$ defines the facet $f_B$, $\kappa_\beta(v)=m_\beta$ if and only if $v \in \bigcap\limits_{B \in \beta} f_B$.

\end{proof}
\begin{cor}\label{c_prekominimumaNbeta}
The following holds:
\[
N_\beta=conv\{v\in \mathcal{V}(P_{\mathcal{B}_\beta}) \mid  \kappa_\beta(v) > m_\beta\}.
\]
\end{cor}

The previous claim offers a comfortable way to obtain the polytope $N_\beta$ from the nestohedron $P_{\mathcal{B}_\beta}$. Now, one is able to handle only with vertices and their coordinates instead of facets and their labels, which is algorithmically closer to Minkowski sums and essentially beneficial with regards to computational aspect.

\begin{exm}
If $n=2$ and $\beta=\bigl\{\{1,2\},\{1\}\bigr\}$, then $\mathcal{B}_\beta=\beta \cup \bigl\{\{2\},\{3\},[3]\bigr\}$ and $P_{\mathcal{B}_\beta}$ is the trapezoid $ABCD$ given in Figure~\ref{s:skica_12}. Since $m_\beta=4$, $F_\beta$ and $N_\beta$ are the vertex $D$ and the triangle $ABC$, respectively.
\end{exm}
\begin{exm}
Let $n=3$.

If $\beta=\bigl\{\{1,2,4\},\{1,2\},\{1\}
\bigr\}$, then $\mathcal{B}_\beta=\beta \cup \bigl \{\{2\},\{3\},\{4\},[4]
\bigr\}$ and  
\linebreak$P_{\mathcal{B}_\beta}=\Delta_{[4]}+\Delta_{\{1\}}
+\Delta_{\{2\}}+\Delta_{\{3\}}+\Delta_{\{4\}}+\Delta_{\{1,2,4\}}+\Delta_{\{1,2\}}$,
the nestohedron $ABCDEFGH$ illustrated in Figure~\ref{s:skica_a4} left. Here, $m_\beta=9$, and hence, $F_\beta$ is the vertex $D$, while $N_\beta$ is the convex hull of the remaining vertices (see Figure~\ref{s:skica_a4} right). Figure~\ref{s:skica_a4rez} depicts the sum $P_{\mathcal{B}_\beta}+N_\beta$, which is normally equivalent to the polytope obtained from $P_{\mathcal{B}_\beta}$ by truncation in the vertex $D$.

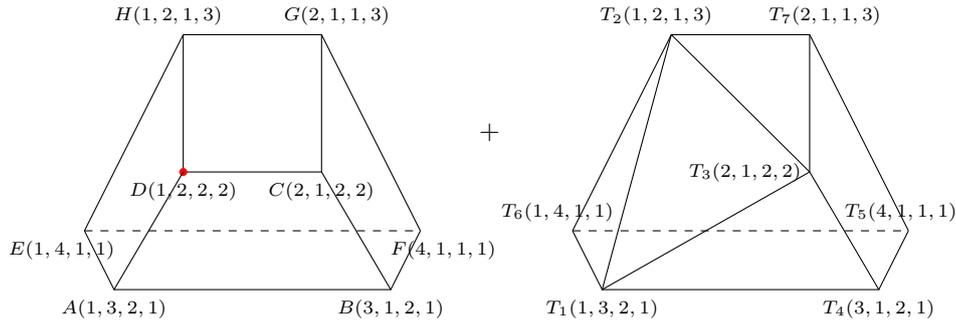
\begin{figure}[h!h!]
\begin{center}
\begin{tabular}{cc}

\begin{tikzpicture}[scale=1.3]

\draw (0,0) node[above, xshift=-0.2cm] {\scriptsize $H(1,2,1,3)$}-- (1.4,0) node[above, xshift=0.2cm] {\scriptsize $G(2,1,1,3)$};
\draw(1.4,0)-- (1.4,-1.4) node[below] {\scriptsize $C(2,1,2,2)$};
\draw (1.4,-1.4)-- (0,-1.4)node[below] {\scriptsize $D(1,2,2,2)$} -- (0,0);

\filldraw[red] (0,-1.4) circle (1pt);;

\draw (1.4,0) -- (2.4,-2) node[below,  xshift=0.3cm] {\scriptsize $F(4,1,1,1)$};
\draw (0,0) -- (-1,-2) node[below,  xshift=-0.3cm] {\scriptsize $E(1,4,1,1)$};
\draw [dashed]
(-1,-2) -- (2.4,-2);

\draw (0,-1.4) -- (-0.7,-2.6)node[below] {\scriptsize $A(1,3,2,1)$} -- (2.1,-2.6) node[below] {\scriptsize $B(3,1,2,1)$}-- (1.4,-1.4);

\draw (-0.7,-2.6) -- (-1,-2);
\draw (2.1,-2.6) -- (2.4,-2);

\draw  (3.1,-1) node {$+$};
\end{tikzpicture}
& \hspace{-0.7cm}
\begin{tikzpicture}[scale=1.3]
\draw (0,0) node[above, xshift=-0.2cm] {\scriptsize $T_2(1,2,1,3)$}-- (1.4,0) node[above, xshift=0.2cm] {\scriptsize $T_7(2,1,1,3)$};
\draw (1.4,0)-- (1.4,-1.4) node[above, xshift=-0.85cm, yshift=-0.25cm] {\scriptsize $T_3(2,1,2,2)$};
\draw (1.4,-1.4) -- (-0.7,-2.6)node[below] {\scriptsize $T_1(1,3,2,1)$}--(0,0);
\draw (0,0) -- (1.4,-1.4);

\draw (1.4,0) -- (2.4,-2) node[above,xshift=-0.1cm  ] {\scriptsize $T_5(4,1,1,1)$};
\draw (0,0) -- (-1,-2) node[above,  xshift=-0.2cm] {\scriptsize $T_6(1,4,1,1)$};
\draw [dashed]
(-1,-2) -- (2.4,-2);

\draw (-0.7,-2.6) -- (2.1,-2.6) node[below] {\scriptsize $T_4(3,1,2,1)$}-- (1.4,-1.4);

\draw (-0.7,-2.6) -- (-1,-2);
\draw (2.1,-2.6) -- (2.4,-2);

\end{tikzpicture}
\end{tabular}
\end{center}
\caption{The polytopes $P_{\mathcal{B}_\beta}$ and $N_\beta$ for $\beta=\bigl\{\{1,2,4\},\{1,2\},\{1\}\bigr\}$} \label{s:skica_a4}
\end{figure}
\vspace{-3mm}
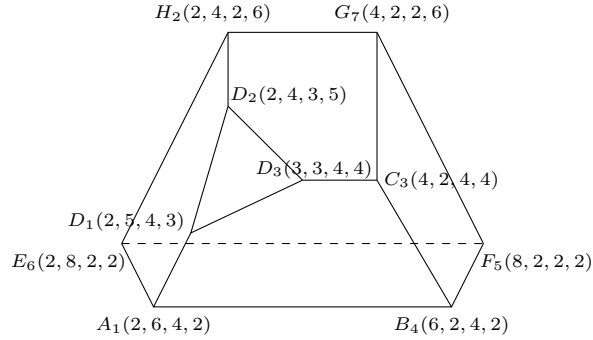
\begin{figure} [h!h!]
\begin{center}
\begin{tikzpicture}[scale=1.4]
\draw(0,0) node[above, xshift=-0.2cm] {\scriptsize $H_2(2,4,2,6)$}-- (1.4,0) node[above, xshift=0.2cm] {\scriptsize $G_7(4,2,2,6)$};
\draw (1.4,0)-- (1.4,-1.4) node[above, xshift=0.85cm, yshift=-0.25cm] {\scriptsize $C_3(4,2,4,4)$};
\draw  (1.4,-1.4)-- (0.7,-1.4) node[above, xshift=0.15cm, yshift=-0.1cm] {\scriptsize $D_3(3,3,4,4)$};
\draw  (0.7,-1.4)-- (-0.35,-1.9) node[above, xshift=-0.85cm, yshift=-0.1cm] {\scriptsize $D_1(2,5,4,3)$}--(0,-0.7) node[above, yshift=-0.1cm, xshift=0.8cm] {\scriptsize $D_2(2,4,3,5)$}--(0,0);
\draw (0,-0.7) -- (0.7,-1.4);

\draw (1.4,0) -- (2.4,-2) node[below,  xshift=0.7cm] {\scriptsize $F_5(8,2,2,2)$};
\draw (0,0) -- (-1,-2) node[below,  xshift=-0.7cm] {\scriptsize $E_6(2,8,2,2)$};
\draw [dashed](-1,-2) -- (2.4,-2);

\draw (-0.35,-1.9) -- (-0.7,-2.6)node[below] {\scriptsize $A_1(2,6,4,2)$} -- (2.1,-2.6) node[below] {\scriptsize $B_4(6,2,4,2)$}-- (1.4,-1.4);

\draw (-0.7,-2.6) -- (-1,-2);
\draw (2.1,-2.6) -- (2.4,-2);

\end{tikzpicture}
\end{center}
\caption{The sum $P_{\mathcal{B}_\beta}+N_\beta$ for $\beta=\bigl\{\{1,2,4\},\{1,2\},\{1\}\bigr\}$} \label{s:skica_a4rez}
\end{figure}
If $\beta=\bigl\{\{1,2,4\},\{1,2\}
\bigr\}$, then $\mathcal{B}_\beta$ and $P_{\mathcal{B}_\beta}$ are the same as in the previous case, while $m_\beta=8$. This minimum is achieved at the points $C$ and $D$ and therefore, $F_\beta$ is the edge $CD$.
Or equivalently, $F_\beta$ is the intersection of the facets labelled by $\{1,2\}$ and $\{1,2,4\}$, i.e.\
the quadrate $CDGH$ and the trapezoid $ABCD$.
However, $N_\beta$ is the convex hull of the remaining points (see Figure~\ref{s:skica_a2} right). Note that the partial sum $P_{\mathcal{B}_\beta}+N_\beta$, depicted in Figure~\ref{s:skica_rez} left, is normally equivalent to the polytope obtained from $P_{\mathcal{B}_\beta}$ by truncation in the edge $CD$.

If $\beta=\bigl\{\{1,2\},\{1\}
\bigr\}$, then $\mathcal{B}_\beta=\beta \cup \bigl\{\{2\},\{3\},\{4\},\{1,2,3\},\{1,2,4\},[4]
\bigr\}$ and
$P_{\mathcal{B}_\beta}=\Delta_{[4]}+\Delta_{\{1\}}
+\Delta_{\{2\}}+\Delta_{\{3\}}+\Delta_{\{4\}}+\Delta_{\{1,2,3\}}+\Delta_{\{1,2,4\}}+\Delta_{\{1,2\}}$,
the nestohedron $ABCDEFGHIJ$ illustrated in Figure~\ref{s:skica_a3} left.
It implies that $m_\beta=4$ and $F_\beta$ is the edge $DJ$. Or equivalently, $F_\beta$ is the intersection of the facets labelled by $\{1,2\}$ and $\{1\}$, i.e.\ the pentagon $AEDJH$ and the quadrate $CDIJ$. However, $N_\beta$ is the convex hull of the remaining points depicted in Figure~\ref{s:skica_a3} right. Notice that, in this case, $N_\beta$ is not simple. Figure~\ref{s:skica_rez} right illustrates the sum $P_{\mathcal{B}_\beta}+N_\beta$, which is normally equivalent to the polytope obtained from $P_{\mathcal{B}_\beta}$ by truncation in the edge~$DJ$.  

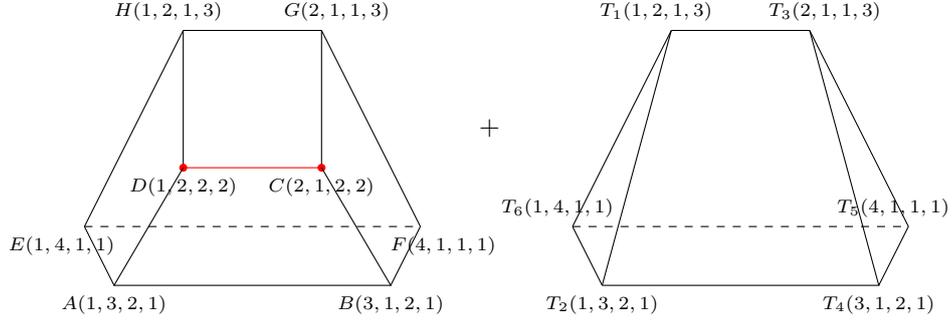
\begin{figure}[h!h!]
\begin{center}
\begin{tabular}{cc}
\begin{tikzpicture}[scale=1.3]

\draw (0,0) node[above, xshift=-0.2cm] {\scriptsize $H(1,2,1,3)$}-- (1.4,0) node[above, xshift=0.2cm] {\scriptsize $G(2,1,1,3)$};
\draw(1.4,0)-- (1.4,-1.4) node[below] {\scriptsize $C(2,1,2,2)$};
\draw[red] (1.4,-1.4)-- (0,-1.4);
\draw (0,-1.4)node[below] {\scriptsize $D(1,2,2,2)$} -- (0,0);

\filldraw[red] (0,-1.4) circle (1pt);
\filldraw[red] (1.4,-1.4) circle (1pt);

\draw (1.4,0) -- (2.4,-2) node[below,  xshift=0.3cm] {\scriptsize $F(4,1,1,1)$};
\draw (0,0) -- (-1,-2) node[below,  xshift=-0.3cm] {\scriptsize $E(1,4,1,1)$};
\draw [dashed]
(-1,-2) -- (2.4,-2);

\draw (0,-1.4) -- (-0.7,-2.6)node[below] {\scriptsize $A(1,3,2,1)$} -- (2.1,-2.6) node[below] {\scriptsize $B(3,1,2,1)$}-- (1.4,-1.4);

\draw (-0.7,-2.6) -- (-1,-2);
\draw (2.1,-2.6) -- (2.4,-2);

\draw  (3.1,-1) node {$+$};
\end{tikzpicture}& \hspace{-0.7cm}
\begin{tikzpicture}[scale=1.3]
\draw(0,0) node[above, xshift=-0.2cm] {\scriptsize $T_1(1,2,1,3)$}-- (1.4,0) node[above, xshift=0.2cm] {\scriptsize $T_3(2,1,1,3)$};

\draw (-0.7,-2.6)node[below] {\scriptsize $T_2(1,3,2,1)$}--(0,0);

\draw (1.4,0) -- (2.4,-2) node[above,xshift=-0.2cm  ] {\scriptsize $T_5(4,1,1,1)$};
\draw
(0,0) -- (-1,-2) node[above,  xshift=-0.2cm] {\scriptsize $T_6(1,4,1,1)$};
\draw [dashed](-1,-2) -- (2.4,-2);

\draw (-0.7,-2.6) -- (2.1,-2.6) node[below] {\scriptsize $T_4(3,1,2,1)$}-- (1.4,0);

\draw (-0.7,-2.6) -- (-1,-2);
\draw (2.1,-2.6) -- (2.4,-2);

 \end{tikzpicture}
\end{tabular}
\end{center}
\caption{The polytopes $P_{\mathcal{B}_\beta}$ and $N_\beta$ for $\beta=\bigl\{\{1,2,4\},\{1,2\}\bigr\}$} \label{s:skica_a2}
\end{figure}
\begin{figure}[h!h!]
\begin{center}
\begin{tabular}{cc}

\begin{tikzpicture}[scale=1.4]
\hspace{-0.15cm}
\draw (-0.2,-0.5) node[above, xshift=-0.2cm] {\scriptsize $H(1,3,1,3)$}-- (1.6,-0.5) node[above, xshift=0.2cm] {\scriptsize $G(3,1,1,3)$}--(1.1,-1.2) node[above, xshift=0.8cm, yshift=-0.3cm] {\scriptsize $I(2,1,2,3)$} -- (1.1,-2.1) node[below, xshift=0.3cm] {\scriptsize $C(2,1,3,2)$}-- (0.3,-2.1)node[below, xshift=-0.3cm] {\scriptsize $D(1,2,3,2)$}-- (0.3,-1.2) node[above, xshift=-0.8cm, yshift=-0.3cm] {\scriptsize $J(1,2,2,3)$};

\draw[red] (0.3,-2.1) -- (0.3,-1.2);
\filldraw [red] (0.3,-2.1) circle (1pt) (0.3,-1.2) circle (1pt);
\draw (-0.2,-0.5) --  (0.3,-1.2);
\draw (1.1,-1.2) --  (0.3,-1.2);

\draw (1.6,-0.5) -- (2.4,-2) node[above,  xshift=0.3cm] {\scriptsize $F(5,1,1,1)$};
\draw (-0.2,-0.5) -- (-1,-2) node[above,  xshift=-0.2cm] {\scriptsize $E(1,5,1,1)$};
\draw [dashed]
(-1,-2) -- (2.4,-2);

\draw (0.3,-2.1) -- (-0.2,-3)node[below] {\scriptsize $A(1,3,3,1)$} -- (1.6,-3) node[below] {\scriptsize $B(3,1,3,1)$}--(1.1,-2.1) ;

\draw (-0.2,-3) -- (-1,-2);
\draw (1.6,-3) -- (2.4,-2);

\draw  (2.8,-1.2) node {$+$};
\end{tikzpicture}
& \hspace{-1.5cm}
\begin{tikzpicture}[scale=1.4]
\draw (1.1,-1.2)--(-0.2,-0.5) node[above, xshift=-0.2cm] {\scriptsize $T_1(1,3,1,3)$}-- (1.6,-0.5) node[above, xshift=0.2cm] {\scriptsize $T_8(3,1,1,3)$}--(1.1,-1.2) node[above, xshift=0.82cm, yshift=-0.3cm] {\scriptsize $T_2(2,1,2,3)$} -- (1.1,-2.1) node[above, xshift=0.4cm] {\scriptsize $T_4(2,1,3,2)$}-- (1.6,-3) node[below] {\scriptsize $T_5(3,1,3,1)$}-- (-0.2,-3)node[below] {\scriptsize $T_3(1,3,3,1)$}--(1.1,-2.1);

\draw (1.6,-0.5) -- (2.4,-2) node[below,  xshift=0.3cm] {\scriptsize $T_7(5,1,1,1)$}--(1.6,-3);
\draw (-0.2,-0.5) -- (-1,-2) node[below,  xshift=-0.3cm] {\scriptsize $T_6(1,5,1,1)$};
\draw [dashed](-1,-2) -- (2.4,-2);

\draw (-0.2,-3) -- (-1,-2);
\draw (-0.2,-0.5) -- (-0.2,-3);

 \end{tikzpicture}
\end{tabular}
\end{center}
\caption{The polytopes $P_{\mathcal{B}_\beta}$ and $N_\beta$ for $\beta=\bigl\{\{1,2\},\{1\}\bigr\}$} \label{s:skica_a3}
\end{figure}
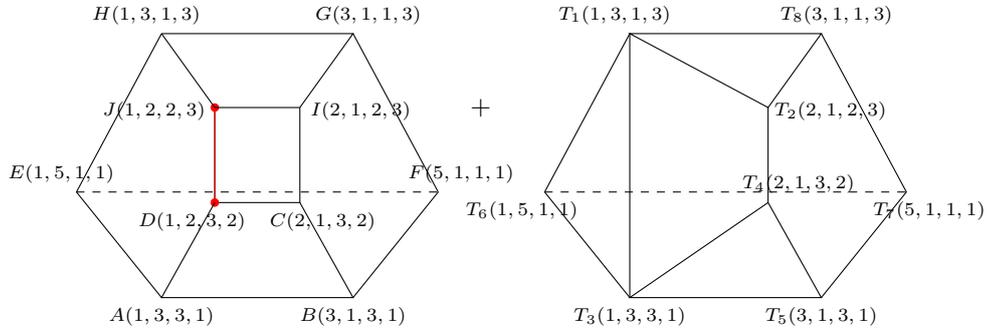
\begin{figure}[h!h!]
\begin{center}
\begin{tabular}{cc}

\begin{tikzpicture}[scale=1.35]
\hspace{-0.15cm}
\draw
(0,0) node[above, xshift=-0.2cm] {\scriptsize $H_1(2,4,2,6)$} -- (1.4,0) node[above, xshift=0.2cm] {\scriptsize $G_3(4,2,2,6)$};

\draw (1.4,0) -- (1.4,-0.8) node[above, xshift=0.85cm, yshift=-0.2 cm] {\scriptsize $C_3(4,2,3,5)$} -- (1.7,-1.8) node[above, xshift=0.85cm, yshift=-0.2 cm] {\scriptsize $C_4(5,2,4,3)$};
\draw (1.7,-1.8)--(-0.3,-1.8)node[above, xshift=-0.85cm, yshift=-0.2 cm] {\scriptsize $D_2(2,5,4,3)$}
 -- (-0.7,-2.6)node[below] {\scriptsize $A_2(2,6,4,2)$} -- (2.1,-2.6) node[below] {\scriptsize $B_4(6,2,4,2)$};

\draw (0,0) -- (0,-0.8) node[above, xshift=-0.85cm, yshift=-0.1 cm] {\scriptsize $D_1(2,4,3,5)$} -- (-0.3,-1.8);
 \draw (0,-0.8) -- (1.4,-0.8) ;
 \draw  (1.7,-1.8) -- (2.1,-2.6);

\draw
(1.4,0) -- (2.4,-2) node[below,  xshift=0.4cm] {\scriptsize $F_5(8,2,2,2)$};
\draw
(0,0) -- (-1,-2) node[below,  xshift=-0.4cm] {\scriptsize $E_6(2,8,2,2)$};
\draw [dashed]
(-1,-2) -- (2.4,-2);

\draw (-0.7,-2.6) -- (-1,-2);
\draw (2.1,-2.6) -- (2.4,-2);

\end{tikzpicture}
& \hspace{-1.5cm}
\begin{tikzpicture}[scale=1.45]
\hspace{-0.15cm}
\draw (-0.2,-0.5) node[above, xshift=-0.2cm] {\scriptsize $H_1(2,6,2,6)$}-- (1.6,-0.5) node[above, xshift=0.2cm] {\scriptsize $G_8(6,2,2,6)$}--(1.1,-1.2) node[above, xshift=0.8cm, yshift=-0.3cm] {\scriptsize $I_2(4,2,4,6)$} -- (1.1,-2.1) node[above, xshift=0.8cm, yshift=-0.35cm] {\scriptsize $C_4(4,2,6,4)$}-- (0.6,-2.1)node[below, xshift=-0.3cm] {\scriptsize $D_4(3,3,6,4)$}-- (0.6,-1.2) node[above, xshift=-0.8cm, yshift=-0.3cm] {\scriptsize $J_2(3,3,4,6)$}--(0.05,-0.85) node[above, xshift=-0.8cm, yshift=-0.3cm] {\scriptsize $J_1(2,5,3,6)$}--(-0.2,-0.5)-- (-1,-2) node[above] {\scriptsize $E_6(2,10,2,2)$};;

\draw (1.1,-1.2) --  (0.6,-1.2);
\draw (1.6,-0.5) -- (2.4,-2) node[above, xshift=-0.2] {\scriptsize $F_7(10,2,2,2)$};

\draw [dashed]
(-1,-2) -- (2.4,-2);

\draw (-1,-2)-- (-0.2,-3)node[below] {\scriptsize $A_3(2,6,6,2)$} -- (1.6,-3) node[below] {\scriptsize $B_5(6,2,6,2)$}--(1.1,-2.1) ;

\draw (1.6,-3) -- (2.4,-2);
\draw (0.05,-0.85) -- (0.05,-2.55) node[above, xshift=-0.8cm, yshift=-0.3cm] {\scriptsize $D_3(2,5,6,3)$}-- (-0.2,-3);
\draw (0.05,-2.55) -- (0.6,-2.1);
\end{tikzpicture}
\end{tabular}
\end{center}
\caption{The sums $P_{\mathcal{B}_\beta}+N_\beta$ for $\beta=\bigl\{\{1,2,4\},\{1,2\}\bigr\}$ and $\beta=\bigl\{\{1,2\},\{1\}\bigr\}$} \label{s:skica_rez}
\end{figure}
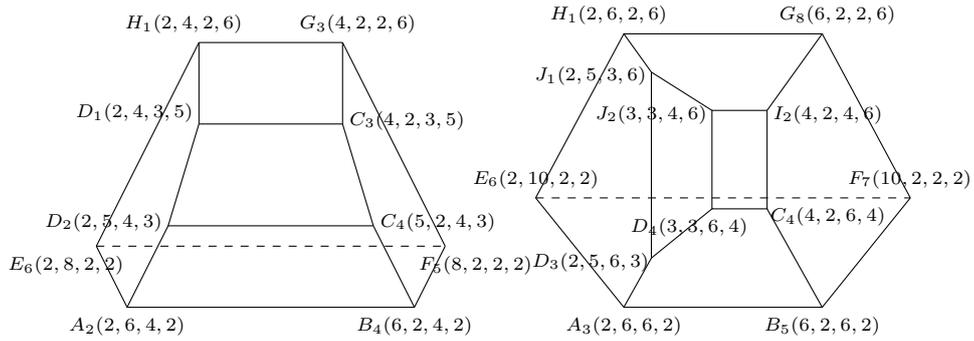

\end{exm}
\begin{exm}
Let $n=4$. If $\beta=\bigl\{\{1,2,3\},\{1,2\}
\bigr\}$, then 
\\$\mathcal{B}_\beta=\beta \cup \bigl \{\{1\},\{2\},\{3\},\{4\},\{5\},\{1,2,3,4\},\{1,2,3,5\},[5]
\bigr\}$ and $\mathcal{V}(P_{\mathcal{B}_\beta})$ is the set
\[\begin{array}{rlllll}
\bigl\{
(6,1,1,1,1), &
(1,6,1,1,1), &
(2,1,5,1,1), &
(1,2,5,1,1), &
(1,2,3,3,1),	
\\[0.7ex]
(4,1,1,3,1), &
(3,1,1,3,2), &
(1,4,1,3,1), &
(2,1,3,3,1), &
(1,2,3,1,3),
\\[0.7ex]
(2,1,2,3,2), &
(1,2,2,3,2), &
(4,1,1,1,3), &
(1,4,1,1,3), &
(1,3,1,2,3),
\\[0.7ex]
(2,1,3,1,3), &
(3,1,1,2,3), &
(2,1,2,2,3), &
(1,2,2,2,3), &
(1,3,1,3,2)  \bigr\}.
\end{array}\]
It implies that $m_\beta=9$ and $F_\beta$ is the quadrilateral whose vertices are the points in the last column from the set. 
Hence, $N_\beta$ is the convex hull of the remaining points.
\end{exm}

Finally, for $n\geqslant 2$ let
\[
PA_{n,1}= \Delta_{[n+1]}+\sum_{\beta\in \mathcal{A}_1} \Delta_{\bigcup\beta}+ \sum_{\beta\in \mathcal{A}_2} N_\beta.
\]
The rest of this section is devoted to a proof of the following result.
\begin{thm}\label{t_glavnat}
$PA_{n,1}$ is an $n$-dimensional Minkowski-realisation of $C$. Moreover,
\[
PA_{n,1} \simeq \mathbf{PA}_n.
\]
\end{thm}

\begin{lem} \label{l_konstrukcije}
If $N\in C_2$ is a maximal nested set which corresponds to a vertex of $F_\beta$, then $N$ is a maximal 0-nested set and $\beta\subseteq N$.
\end{lem}

\begin{proof}
Let $v$ be a vertex of $F_\beta$ that corresponds to $N$. By Proposition~7.5 in \cite{P09}, $\lvert N \rvert =n$. Since $v\in f_B$ for every $B \in \beta$, we have that $\beta \subseteq N$.
In particular, if $k=n$, then $N=\beta$.  Otherwise, by Definition~\ref{d_nested_skup_u_odnosu_na_bilding_skup}, $N$ is obtained by enlarging $\beta$ with $n-k$ elements of $\mathcal{B}_\beta-[n+1]$ such that the union of every $N$-antichain belongs to $C_2-\mathcal{B}_\beta$.

Let $B_1$ and $B_2$ be two non-singleton elements of $\mathcal{B}_\beta-\{[n+1]\}$, which are contained in $N$ and incomparable. Then $\{\{B_1\},\{B_2\}\}$ is an $N$-antichain such that $B_1\cup B_2\supseteq \beta_{max}$ or $\beta_{min} \supseteq B_1\cup B_2$. Therefore, $N$ does not contain such a pair.

Then, let us suppose that $N$ contains two or more singletons. For every singleton $B=\{i_{j}\}$, where $i_{j}\in \beta_{max}-\beta_{min}$, there is $ A \in \beta$, such that $\{\{A\},\{B\}\}$ is an $N$-antichain whose union belongs to $\beta$.
Similarly, for every singleton $B=\{i_{j}\}$, where $i_{j}\in [n+1]-\beta_{max}$, we have that $\{\{\beta_{max}\},\{B\}\}$ is an $N$-antichain whose union has $\beta_{max}$ as a subset, and therefore, belongs to the building set.
It remains to check singleton subsets of $\beta_{min}$.
The union of every pair of singleton subsets of $\beta_{min}$, is also a subset of $\beta_{min}$, i.e.\ belongs to the building set. Then, just one of the singleton subsets of $\beta_{min}$ can belong to $N$ and all the other subsets of $N$ are mutually comparable.
\end{proof}  
Finally, let
  $\pi_{\beta,c}$ be the hyperplane
$$x_{i_1}+2x_{i_2}+\ldots+k(x_{i_k}+\ldots+x_{i_{k+l}})=m_\beta+c,$$
%and let $a_{\beta}$ be an outward normal to the halfspaces $\pi_{\beta,c}^{\geqslant}$.
where $c\in (0,1]$, and $a_{\beta}$ be an outward normal to the halfspaces $\pi_{\beta,c}^{\geqslant}$. 
\begin{rem}\label{r_haovi}
For every $c\in (0,1]$, the sum of outward normal vectors to the facets that contains $F_\beta$ is an outward normal vector to the halfspace $\pi_{\beta,c}^{\geqslant}$.  
\end{rem}

\begin{lem}\label{l_glavnal}
For an element
$\beta$ of $\mathcal{A}_2$, a polyope
$$P_{\mathcal{B}_\beta}\bigcap {\pi_{\beta,c}}^\geqslant$$ is an $F_\beta$-deformation of $P_{\mathcal{B}_\beta}$.
Moreover, the following holds:
\[
    P_{\mathcal{B}_\beta}\bigcap {\pi_{\beta,c}}^\geqslant =
\begin{cases}
   \emph{a parallel truncation }\emph{tr}_{F_\beta}P_{\mathcal{B}_\beta},& c\in (0,1)\\
   N_\beta,              & c=1.
\end{cases}
\]

\end{lem}

\begin{proof}
Firstly, recall that each coordinate of some vertex of $P_{\mathcal{B}_\beta}$ is a natural number. This, together with Proposition~\ref{p_prekominimumaFbeta} and Corollary~\ref{c_prekominimumaNbeta}, implies that ${\pi_{\beta,c}}^\geqslant$ is beyond every vertex of $F_\beta$ and beneath every vertex of $N_\beta$, for every $c \in (0,1)$. Hence, $P_{\mathcal{B}_\beta}\bigcap {\pi_{\beta,c}}^\geqslant$ is a truncation of $P_{\mathcal{B}_\beta}$ in its face $F_\beta$. 
Since $\pi_{\beta,0}$ defines $F_\beta$, all these truncations are parallel, and hence, all of them are $F_\beta$-deformations of $P_{\mathcal{B}_\beta}$ (see Example~\ref{e_trunkacijajedeformacija}).
Also, 
\[
P_{\mathcal{B}_\beta}\bigcap {\pi_{\beta,1}}^\geqslant =N_\beta.
\]

Let $\{U,W\}$ be a partition of the set $\mathcal{V}(N_\beta)$ such that all the elements of $U$ are adjacent to $F_\beta$ in $P_{\mathcal{B}_\beta}$. To be precise, $u\in U$ if and only if there exists $v\in \mathcal{V} (F_\beta)$ such that $u$ and $v$ are adjacent in $P_{\mathcal{B}_\beta}$.
Let $u=(u_1,\dots,u_{n+1})$ be an element of $U$ and $v=(v_1,\dots,v_{n+1})$ be a vertex of $F_\beta$ adjacent to $u$ in $P_{\mathcal{B}_\beta}$. There are two maximal nested sets $N_v,N_u\in C_2$ corresponding to $v$ and $u$, respectively. By Lemma~\ref{l_konstrukcije}, $N_v$ is a maximal 0-nested set containing $\beta$ as a subset.
Since $\lvert N_v \rvert =\lvert N_u \rvert =n$ (see \cite[Proposition~7.5]{P09}) and $u$, $v$ are adjacent, we have that $\lvert N_v \cap N_u \rvert =n-1$, which entails that $N_u$ can be obtained from $N_v$ by substituting an element $S_v$ for another element $S_u$ of $(\mathcal{B}_\beta-[n+1])-\beta$. Since $\beta \nsubseteq N_u$,
$S_v \in \beta$. 
Moreover, following the proof of Lemma~\ref{l_konstrukcije}, we can verify that $S_u$ is a singleton. We conclude that for two distinct vertices of $F_\beta$, there is no element of $U$ adjacent to both of them.

Now, we show that $u\in \pi_{\beta,1}$, for every $u \in U$. 
Let $N_v$ be
\[
\bigl\{\{i_{n},\ldots,i_1\},\dots,\{i_{n},i_{n-1}\},\{i_n\}
\bigr\}
\]
such that $\beta\subseteq N$.
From $v\in F_\beta$ and $u\notin F_\beta$, we have that $$
m_\beta(v)=v_{i_1}+2v_{i_2}+\ldots+k(v_{i_k}+\ldots+v_{i_{k+l}})=m_\beta$$
and
$$m_\beta(u)=u_{i_1}+2u_{i_2}+\ldots+k(u_{i_k}+\ldots+u_{i_{k+l}})>m_\beta .$$
For every element $A$ of the set $N_v \cap N_u=N_v-\{S_v\}=N_u-\{S_u\}$,  we have that $u,v \in f_A$. This, together with Proposition~\ref{p_burstabernejednakosti}, entails the following:
\[\begin{array}{lr}
(\star) &
\sum\limits_{i\in A}u_i=\sum\limits_{i\in A} v_i=\lvert \mathcal{B}_{\beta \vert A} \rvert.
\end{array}\]
From Proposition~\ref{p_burstabernejednakosti}, we also have
\[\begin{array}{lr}
(\star\star) &  u_1+\ldots+u_{n+1}=v_1+\ldots+v_{n+1}=\lvert \mathcal{B}_{\beta \vert [n+1]} \rvert .
\end{array}\]
Let $S_v=\{i_{k+l},\ldots,i_{p}\}$, where $1\leqslant p \leqslant k$. Then, by $(\star)$,
\[
m_\beta(u)=m_\beta-(v_{i_p}+v_{i_{p+1}}+\ldots+v_{i_{k+l}})+(u_{i_p}+u_{i_{p+1}}+\ldots+u_{i_{k+l}}).
\]
Let us analyse all possible cases.

\vspace{2ex}

(1) If $S_v$ is a singleton, i.e.\ $p=k=n$ and $l=0$, then
\[
m_\beta(u)= m_\beta-v_{i_k}+u_{i_k}.
\]
Since $\{i_k\}\in N_v$, by Proposition~\ref{p_burstabernejednakosti}, $v_{i_k}=\lvert  \mathcal{B}_{\beta \vert \{i_k\}} \rvert =1$.
Following the proof of Lemma~\ref{l_konstrukcije}, one may verify that $N_u$ is a nested set if and only if $S_u=\{i_{k-1}\}$. This entails that $u \in f_{\{i_{k-1}\}}$, i.e.\ applying Proposition~\ref{p_burstabernejednakosti},
$$
u_{i_{k-1}}=\lvert \mathcal{B}_{\beta \vert \{i_{k-1}\}}\rvert =1.
$$
Having that $\{i_k,i_{k-1}\}\in N_u \cap N_v$ and applying Proposition~\ref{p_burstabernejednakosti}, we obtain
\[
u_{i_k}=\lvert \mathcal{B}_{\beta \vert \{i_k,i_{k-1}\}} \rvert -u_{i_{k-1}}=3-1=2.
\]
Therefore, $m_\beta(u)= m_\beta-1+2=m_\beta+1$.

\vspace{2ex}

(2) If $S_v$ is not a singleton, then $\{i_{k+l},\ldots,i_{p+1}\} \in N_v\cap N_u$. Applying $(\star)$, we obtain
\[
u_{i_{p+1}}+\ldots+u_{i_{k+l}}=v_{i_{p+1}}+\ldots+v_{i_{k+l}},
\]
which implies
\[
m_\beta(u)=m_\beta-v_{i_p}+u_{i_p}.
\]
Having that $S_v,\{i_{k+l},\ldots,i_{p+1}\}\in N_v$ and applying Proposition~\ref{p_burstabernejednakosti}, we get the following equations:
\[
v_{i_p}=\lvert \mathcal{B}_{\beta \vert S_v} \rvert -\lvert \mathcal{B}_{\beta \vert S_v-\{i_p\}}\rvert =\lvert \mathcal{B}_{\beta \vert  S_v-\{i_p\}}\rvert +2-\lvert \mathcal{B}_{\beta \vert S_v-\{i_p\}}\rvert =2.
\]

\vspace{1ex}

(2.1) If $\lvert S_v \rvert \neq n$, then we follow the proof of Lemma~\ref{l_konstrukcije} in order to analyse the form of $N_u$. In that manner, we conclude that $N_u$ is a nested set if and only if $S_u=\{i_{p-1}\}$. This entails that $u \in f_{\{i_{p-1}\}}$, i.e.\ applying Proposition~\ref{p_burstabernejednakosti},
$$
u_{i_{p-1}}=\lvert \mathcal{B}_{\beta \vert \{i_{p-1}\}} \rvert =1.
$$
Having that $\{i_{k+l},\ldots,i_{p-1}\},\{i_{k+l},\ldots,i_{p+1}\} \in N_v \cap N_u$ and applying $(\star)$, we obtain the following equations:
\[
u_{i_{k+l}}+\ldots+u_{i_{p-1}}=\lvert  \mathcal{B}_{ \beta \vert S_v\cup \{i_{p-1}\}} \rvert =\lvert \mathcal{B}_{\beta \vert S_v} \rvert  +2,
\]
\[
u_{i_{k+l}}+\ldots+u_{i_{p+1}}=\lvert  \mathcal{B}_{\beta \vert S_v - \{i_p\}} \rvert =\lvert \mathcal{B}_{\beta \vert S_v} \rvert  -2.
\]
Hence, $
u_{i_{p}}+u_{i_{p-1}}=4$, and therefore,
$$m_\beta(u)=m_\beta-2+(4-1)=m_\beta+1.$$

\vspace{1ex}

(2.2) If $\lvert S_v \rvert = n$, i.e.\ $p=1, k+l=n$, then we again follow the proof of Lemma~\ref{l_konstrukcije} and conclude that $N_u$ is a nested set if and only if $S_u=\{i_{n+1}\}$. Therefore, $u \in f_{\{i_{n+1}\}}$, i.e.\ applying Proposition~\ref{p_burstabernejednakosti},
$$
u_{i_{n+1}}=\lvert \mathcal{B}_{\beta \vert \{i_{n+1}\}} \rvert =1.
$$
Having that $\{i_{n},\ldots,i_{2}\} \in N_v \cap N_u$ and applying $(\star)$, we obtain that
\[
u_{i_{n}}+\ldots+u_{i_{2}}=\lvert \mathcal{B}_{\beta \vert S_v - \{i_{1}\}} \rvert =\lvert  \mathcal{B}_{\beta \vert  S_v} \rvert -2=\lvert  \mathcal{B}_{\beta \vert [n+1]} \rvert -4.
\]
This, together with $(\star\star)$, entails that
\[
u_{i_{n+1}}+u_{i_{1}}=\lvert \mathcal{B}_{\beta \vert [n+1]} \rvert -\bigl(\lvert \mathcal{B}_{\beta \vert [n+1]} \rvert -4 \bigr)=4.
\]
Hence, $m_\beta(u)=m_\beta-v_{i_1}+u_{i_1}=
m_\beta-2+(4-1)=m_\beta+1$.

\vspace{2ex}

All this entails that for every $u \in U$, $m_\beta(u)=m_\beta+1$, i.e.\ $conv U \in \pi_{\beta,1}$. 
%implies that
%\[\lvert U \rvert= k\lvert \mathcal{V}(F_\beta) \rvert \geqslant n.\]
We can conclude that $N_\beta \cap \pi_{\beta,1}$ is exactly $conv U$. Otherwise, i.e.\ if there would exist an element $w\in W$ contained in $\pi_{\beta,1}$, then since $w$ is not adjacent to $F_{\beta}$, this vertex of $P_{\mathcal{B}_\beta}$ would be contained in the convex hull of $U$, which would be contradiction. Also, note that $\mathcal{V}(conv U)=U$, because each element of $U$ is a vertex of $P_{\mathcal{B}_\beta}$. Therefore, $conv U$ is an $(n-1)$-face of $N_\beta$. Let $c$ be an arbitrary element of the interval $(0,1)$ and let us denote by $f$ the facet of the truncation $P_{\mathcal{B}_\beta} \cap \pi_{\beta,c}^{\geqslant}$ contained in the truncation hyperplane. Since we have already concluded that for two distinct vertices of $F_\beta$ there is no element of $U$ adjacent to both of them, we now can conclude that for two distinct vertices of $f$ there is no element of $U$ adjacent to both of them in the truncation. Hence, since $\pi_{\beta, c}$ and $\pi_{\beta,0}$ are parallel, $conv U$ is a translate of $f$. In other words, 
%Since $\mathcal{V}(N_\beta)\subset \mathcal{V}(P_{\mathcal{B}_\beta})$, all the other facets of $N_\beta$ are the facets of $P_{\mathcal{B}_\beta}$. 
$N_\beta$ can be obtained from the truncation by parallel translation of the facet $f$ without crossing the vertices. According to Remark~\ref{r_deformacijagruboreceno}, $N_\beta$ is an $F_\beta$-deformation of $P_{\mathcal{B}_\beta}$.
\end{proof}

\begin{rem}\label{r_dimenzijaNbeta}
According to Definition~\ref{d_nested_skup_u_odnosu_na_bilding_skup}, the set
\[
\bigl\{\{i_1\},\ldots,\{i_{n+1}\} \bigr\}-\bigl\{\{i_n\} \bigr\}
\]
is a maximal nested set, which corresponds to some element of the set $W$ (defined in the previous proof). Hence, $W\neq \emptyset$, i.e.\ $N_\beta$ is an $n$-polytope with the facet $conv U\in \pi_{\beta,1}$. 
\end{rem}

%\begin{lem}\label{l_odnosmedjudeformacijama}
%Let $T=\emph{tr}_{F_k}(\ldots \emph{tr}_{F_1}P_1)$ be a polytope obtained from the polytope $P_1\in \mathcal{M}_n$ by the sequence of $k$ parallel truncations in its faces $F_1,\ldots,F_k$. If $F$ is a face of $P_1$ such that $F\neq F_i$ for every $i \in[k]$, and $P_2\in \mathcal{M}_n$ is an $F$-deformation of $P_1$, then $P_2$ is an $F$-deformation of $T$.
%\end{lem}
%\begin{proof}
%\end{proof}

\begin{lem}\label{l_normalezaglavnut}
Let $A=\{a_1,\ldots,a_n\}$ be the spanning set of vectors for an $n$-cone in $\mathbf{R}^n$, and let $h_I$ be the vector defined as 
$$h_I=\sum\limits_{i\in I} a_i,$$ 
where $I\subseteq [n]$ and $\lvert I \rvert \geqslant2$. The following claims hold.
\begin{enumerate}

\item[\emph{(i)}] For every two subsets $I,J \subseteq [n]$ such that $I\subset J$, the vector $h_J$ is contained in the cone spanned by the set $\{h_I\}\cup\{a_i \mid i\in J-I\}$.  

\item[\emph{(ii)}] 
For $2 \leqslant m\leqslant n-1$, let $I_1,\ldots,I_{m} \subseteq [n]$ such that $I_1 \supset I_2 \supset\ldots \supset I_{m}$, and let $A_1$ be the set obtained from $A$ by replacing $m$ elements with the vectors $h_{I_1},\ldots,h_{I_{m}}$. If $A_1$ spans an $n$-cone $N_1$, 
then for every $1\leqslant k<m$ 
there is exactly one element $i\in I_k-I_{k+1}$ such that the set obtained from $A_1$ by replacing $h_{I_k}$ with $a_{i}$, spans an $n$-cone $N_2$ which contains $N_1$.

\end{enumerate}
\end{lem}
\begin{proof}
The first claim follows directly from the fact that $h_J=h_I+\sum\limits_{i \in J-I}a_i$. For $1\leqslant k<m$, let $\Delta_k=I_k-I_{k+1}$. Since $I_1 \supset I_2 \supset\ldots \supset I_{m}$, the sets $\Delta_k$ are mutually disjoint. 
%we may assume that $I_k-I_{k+1}=[t]-[p]=\{p+1,\ldots,t\}$. 
Since $h_{I_m}$ is a spanning ray of $N_1$, there is at least one element of the set $\{a_i\mid i \in I_m\}$ which is not contained in $A_1$. Also, at least one element of the set $\{a_i\mid i \in \Delta_k\}$ does not belong to $A_1$; otherwise, by the claim (i), $h_{I_k}$ would not be a spanning ray of $N_1$. If we suppose that for some $k$ there are two or more such elements, then there is more than $m$ elements of $A$ that are not contained in $A_1$. This contradicts the assumption $\lvert A-A_1\rvert=n-m$. Using the claim (i), it remains to conclude that $N_2$ is an $n$-cone which contains each of the spanning vectors of $N_1$, i.e.\ which contains $N_2$. 
\end{proof}

\begin{proof} [Proof of Theorem~\ref{t_glavnat}]
%For the function $\varphi: \mathcal{B}_1\rightarrow \mathcal{M}_{n+1}$ be the function

%\[
%    \varphi(\beta)=
%\begin{cases}
 %   \Delta_{\bigcup\beta},& \text{if } \beta\in \mathcal{A}_1\\
 %   N_\beta,              & \text{otherwise}.
%\end{cases}
%\]
%meni je potpuno nebitna injektivnost za rad
%For an element $\beta$ of $\mathcal{A}_1$, we have that $\beta=\{B\}$ for some $B\subset [n+1]$. Hence, $\Delta_{\bigcup\beta}$ is
%the $(\lvert B \rvert -1)$-simplex $\Delta_B$ and it is easy to see that
%$$\Delta_{\bigcup\beta_1}=\Delta_{\bigcup\beta_2} \Rightarrow\beta_1=\beta_2,$$
%for $\beta_1,\beta_2 \in \mathcal{A}_1$. Also, if we assume that for two elements $\beta_1,\beta_2\in \mathcal{A}_2$ $N_{\beta_1}=N_{\beta_2}$ holds, then, by Lemma~\ref{l_glavnal} and Definition~\ref{d_pideformacija}(i), $F_{\beta_1}=F_{\beta_2}$, i.e.\ $\beta_1=\beta_2$. Since $N_\beta$ is an $n$-polytope (see Remark~\ref{r_dimenzijaNbeta}), we conclude that $\varphi$ is injective. Hence, $PA_{n,1}$ is well formed according to Defnition~\ref{d_Mink_realizacija}.
It is obvious that $PA_{n,1}$ is well formed according to Defnition~\ref{d_Mink_realizacija}(ii). Since each summand is either $d$-simplex $\Delta_{\bigcup\beta}$, $d<n$, or $n$-polytope $N_{\beta}$ (see Remark~\ref{r_dimenzijaNbeta}), applying Corollary~\ref{prop_minkowski_svojstva2}(ii) we conclude that the sum $PA_{n,1}$ is at least $n$-dimensional.

Let us consider the partial sum
\[
S_0=\Delta_{[n+1]}+\sum_{\beta\in \mathcal{A}_1} \Delta_{\bigcup\beta}.
\]
It is an $n$-permutohedron (see the end of Section~3). More precisely, by Proposition~\ref{p_burstabernejednakosti} (or \cite[Theorem~1.5.4]{BP15}), $S_0$ is the intersection of the following $l=2^{n+1}-2$ halfspaces
\[
\alpha_j^{\geqslant}: \langle a_j,x\rangle\geqslant b_j,  \;   \; 1\leqslant j\leqslant l, 
\]
where for every $j \in[l]$ there is $B\subset[n+1]$ such that $\alpha_j^{\geqslant}$ is the halfspace
\[
H_{B}^{\geqslant}=\bigl\{x \in \mathbf{R}^{n+1} \mid \sum_{i\in B} x_i\geqslant 2^{\lvert B \rvert }- 1\bigr\}.
\]
Each of these halfspaces is facet-defining, i.e.\ determines the facet $f_B$ of $S_0$. According to Definition~\ref{d_Mink_realizacija}, $S_0$ is an $n$-dimensional Minkowski-realisation of the simplicial complex $C_1$.
Let us relabel its facets by the corresponding elements of $\mathcal{A}_1$--the facet $f_B$ is labelled by $\{B\}$. By Corollary~\ref{c_labele_odPA_n}, each of them is parallel to the equilabelled facet of $\mathbf{ PA}_n$.  

Before we show that indexed set of the remaining summands is a truncator set of summands for this permutohedron, we firstly show that for an arbitrary $\beta \in \mathcal{A}_2$, $N_{\beta}$ is a truncator summand for $S_0$.

Recall that we can obtain $S_0$ by a sequence of parallel truncations of the nestohedron $P_{\mathcal{B}_{\beta}}$, up to normal equivalence. In other words, $P_{\mathcal{B}_{\beta}}$ can be obtained from $S_0$ by a sequence of parallel translations of the facets without crossing over the vertices. Formally, $P_{\mathcal{B}_{\beta}}$ can be defined as the intersection of the halfspaces
\[
\gamma_j^{\geqslant}:\langle a_j,x\rangle\geqslant c_j, \; 1 \leqslant j \leqslant l, 
\]
such that 
%the first $k$ of them are facet-defining, and 
for every vertex of $S_0$ which is the intersection of the hyperplanes $\alpha_j$, $j \in J\subset [l]$, the intersection of the hyperplanes $\gamma_j$, $j \in J$, is a vertex of $P_{\mathcal{B}_{\beta}}$. 
Since $F_{\beta}$ is the intersection of the facets of $P_{\mathcal{B}_{\beta}}$ indexed by the elements of $\beta$ which are mutually comparable, there exists the corresponding the same dimensional face $F$ of $S_0$ (the intersection of the facets indexed by the same elements), i.e.\ for every facet of $P_{\mathcal{B}_{\beta}}$ containing $F_\beta$, there is the corresponding facet of $S_0$ containing $F$ with the same outward normals. Then, applying Lemma~\ref{l_glavnal} for some \mbox{$c\in (0,1)$}, we conclude that there is a parallel truncation tr$_FS_0=S_0 \cap \alpha_0^{\geqslant}$ such that $a_\beta$ is an outward normal to $\alpha_0^{\geqslant}$. 
This, together with Definition~\ref{d_pideformacija}(ii) and the fact that $N_{\beta}$ is an $F_{\beta}$-de\-for\-ma\-tion of the nestohedron, implies that $N_{\beta}$ can be obtained from tr$_FS_0$ by parallel translations of the facets without crossing over the vertices. Also, by Definition~\ref{d_pideformacija}(i), we have that $S_0\cap \alpha_0 \simeq P_{\mathcal{B}_{\beta}} \cap \pi_{\beta,c} \simeq N_{\beta} \cap \pi_{\beta,1}$. Therefore, according to Definition~\ref{d_pideformacija}, $N_{\beta}$ is an $F$-deformation of $S_0$, and hence,
%by Lemma~\ref{l_deformacijajesingleprofinjenje}, $S_0+P_1\simeq$ tr$_FS_0$ and $P_{1}$ is a truncator summand for $S_{0}$, according to Definition~\ref{d_summand_produces_a_facet}. 
by Proposition~\ref{l_deformacijajesingleprofinjenje}, $N_{\beta}$ is a truncator summand for $S_0$ in $F$, i.e.\ $S_0+N_{\beta} \simeq \text{tr}_FS_0$.

Now, for $m=\lvert \mathcal{A}_2, \rvert $ let $x: [m]\longrightarrow \mathcal{A}_2$ be an indexing function such that $\lvert x(i) \rvert  \geqslant \lvert x(j) \rvert$ for every $i<j$. Then, let $\{Q_i\}_{i \in [m]}$ be an indexed set of polytopes such that $Q_i=N_{x(i)}$.
We show that this indexed set is a truncator set of summands for the permutohedron
$S_0$, which entails that Definition~\ref{d_Mink_realizacija}(iii) is satisfied. 

Starting from the permutohedron $S_0$, let $S_1$ be the partial sum $S_0+Q_1$, and for the sake of simplicity, let $\beta$ denotes $x(1)$. From the above conclusion, we have that $S_1 \simeq \text{tr}_FS_0$, where tr$_FS_0$ is a parallel truncation in the face that corresponds $F_\beta$ and $a_{\beta}$ is an outward normal to the truncation halfspace.
 
We iteratively repeat the following for every $2\leqslant i \leqslant m$. To be precise, at $i$th step, let $S_{i}=S_{i-1}+Q_{i}$ and suppose that for every $j<i$ we have that \mbox{$S_{j}\simeq \text{tr}_FS_{j-1}$}, where tr$_FS_{j-1}$ is a parallel truncation in the face that corresponds to $F_{x(j)}$. Again, let $\beta=x(i)$. 
%There is no facet of $S_{i-1}$ defined by a halfspace parallel to ${\pi_{x(i),1}}^\geqslant $. 
As long as the cardinality of $\beta$ is maximal, i.e.\ as long as $F_{\beta}$ and the corresponding face $F$ of $S_{i-1}$ are vertices, 
%of the corresponding nestohedron and $S_{i-1}$, respectively,
we may apply completely analogous reasoning as above and obtain that $Q_i$ is an $F$-de\-for\-mation of $S_{i-1}$, which, together with Proposition~\ref{l_deformacijajesingleprofinjenje}, implies $S_{i}\simeq \text{tr}_FS_{i-1}$. Since $\mathcal{B}_1$ contains all maximal 0-nested sets, we can notice that each vertex of $S_0$ is truncated. 
%To be precise, $S_{i-1}$ can be presented as the intersection of the following facet-defining halfspaces
%\[\delta_j^{\geqslant}: \langle a_j,x\rangle\geqslant d_j,  \;   \; 1\leqslant j\leqslant l+i-1. \]

Suppose that $k=\dim(F_\beta)>0$. Since all truncations at the previous steps were in faces of lower dimensions, there exists the corresponding face $F$ of $S_{i-1}$. Also, since all truncations were parallel and the normal equivalences held, $P_{\mathcal{B}_\beta}$
%can be defined as the intersection of the following halfspaces
%\[\sigma_j^{\geqslant}: \langle a_j,x\rangle\geqslant e_j,  \;   \; 1\leqslant j\leqslant l+i-1,\] 
still can be obtained from $S_{i-1}$ by parallel translations of the facets without crossing over the vertices, but the second part of Definition~\ref{d_pideformacija}(i) does not hold generally. It means that we can not conclude that $N_\beta$ is an $F$-deformation of $S_{i-1}$. 
%The sum $S_{i-1}$ can be obtained from $S_0$ by a sequence of parallel truncations, up to normal equivalence. Therefore, by completely analogous reasoning as in the case $i=0$, we conclude the following. There is the corresponding face $F$ of $S_{i-1}$ and the halfspace $\pi$ with the same outward normal as $\pi_{x(i),1}$ such that tr$_FS_{i-1}=S_{i-1}\cap {\pi}^\geqslant$ is a parallel truncation. 
%However, $S_{i}\simeq$ tr$_FS_{i-1}$ still holds and we prove that by showing that each 
Thus, in order to prove that $S_{i}\simeq \text{tr}_FS_{i-1}$ still holds, we use Proposition~\ref{p_dovoljnomakskonuseposmatrati}. Since $Q_i$, $S_{i-1}$ and tr$_FS_{i-1}$ are $n$-polytopes, without lose of generality, we consider the union of all normal cones in each of their fans as $\mathbf{R}^n$.    
%we show that if $N=N_{v}(prvi)\cap N_{v_2}(drugi)$ is a maximal cone for arbitrary $v\in mathcal{V}(prvi)$ and $v\in mathcal{V}(prvi)$, then $N \in \mathcal(N)(tr)$ and there is no maximal normal cone in ... which can not be obtained in that manner.
%the sets of maximal normal cones in $\mathcal{N}(S_{i-1}+Q_i)$ and $\mathcal{N}(\text{tr}_FS_{i-1})$ are equal. 
%let us compare the fans $\mathcal{N}(S_{i-1})$ and $\mathcal{N}(Q_i)$ with the fan $\mathcal{N}(S_0)$.

Firstly, we consider all normal cones to $Q_i$ at vertices not contained in $\pi_{\beta,1}$. Let $N$ be one of them. Since $Q_i$ is an $F'$-deformation of $S_0$, where $F'$ is the corresponding face of $S_0$, by Lemma~\ref{l_zvezda}, $N$ is the union of the normal cones to tr$_{F'}S_0$ at vertices not contained in the truncation hyperplane, which are the normal cones to $S_0$ at vertices not contained in $F'$. 
%, such that the normal cones to tr$_F'S_0$ at the vertices contained in the truncation hyperplane u razlicitom max konusu od $Q_i$, tj. jedini zrak koji je u $Q_i$ a nije u $S_0$ je bas $a_\beta$.
Since $S_{i-1}$ is obtained from $S_0$ by the sequence of parallel truncations, up to normal equivalence, by Lemma~\ref{l_konusiparalelnetrunkacije}, $N$ is the union of the normal cones $N_1,\ldots,N_t$ to $S_{i-1}$ at vertices not contained in $F$. Then, for every $i\in [t]$, $N \cap N_i$ is $N_i$, the normal cone to tr$_FS_{i-1}$ at a vertex not contained in the truncation hyperplane.
%and the intersection of $N$ and any other maximal normal cone to $S_{i-1}$ is not a maximal cone. Also, since each $N_i$ is contained in $N$, its intersection with the other maximal normal cones in the same fan is not a maximal cone, and every maximal normal cones to tr$_FS_{i-1}$ at a vertex not contained in the truncation hyperplane can be obtained as described intersection. 
Therefore, Proposition~\ref{p_dovoljnomakskonuseposmatrati}(ii) is satisfied for the considered maximal normal cones to $Q_i$.

Now, let $N$ be the normal cone to $Q_i$ at a vertex contained in $\pi_{\beta,1}$. By Lemma~\ref{l_zvezda}, $N=N'\cup N_v$, where $N_v$ is the normal cone to tr$_{F'}S_0$ at a vertex contained in the truncation hyperplane, while $N'$ is the union of the normal cones to tr$_{F'}S_0$ at vertices not contained in that hyperplane, i.e.\ the union of the same normal cones to $S_0$ at vertices not contained in $F'$. As above, $N'$ is the union on maximal normal cones to $S_{i-1}$ which are the normal cones to tr$_{F}S_{i-1}$ at vertices not contained in the truncation hyperplane, and hence, their intersections with $N$ are these cones themselves. 
%It implies that Proposition~\ref{p_dovoljnomakskonuseposmatrati}(i) is satisfied for each normal cone to tr$_FS_{i-1}$ at a vertex not contained in the truncation hyperplane (for more details see the proof of Proposition~\ref{l_deformacijajesingleprofinjenje}).

Any other $n$-cone, which can be obtained as the intersection of $N$ and some maximal normal cone to $S_{i-1}$, is the intersection of that cone and $N_v$.  
In order to analyse such cases, let $u$ be a vertex of $S_0$ contained in $F'$. Since $S_0$ is simple polytope, its $k$-face $F'$ belongs to exactly $p=n-k$ facets. Without lose of generality, we assume that they are defined by the halfspaces $\{\alpha_j^{\geqslant} \mid j \in [p]\}$, while $u$ is the intersection of the hyperplanes $\{\alpha_j\mid j \in [n]\}$. Then, every element of the set $\{-a_j \mid j \in [p]\}$ is a spanning ray of the normal cone to $S_{i-1}$ at a vertex contained in $F$.
Since $F'$ is also simple, there are exactly $k$ vertices adjacent to $u$ in $F'$, which implies that there are exactly $p$ vertices of $S_0$ adjacent to $u$ that do not belong to $F'$. Hence, by Lemma~\ref{l_konusiparalelnetrunkacije}, $N_u(S_0)$ is the union of $p$ normal cones $N_{v_1},...,N_{v_p}$ to tr$_{F'}S_0$ at the corresponding vertices contained in the truncation hyperplane. Exactly their intersections with an arbitrary maximal normal cone $M$ to $S_{i-1}$ remain to be considered. 
In order to prove that these intersections also satisfy Proposition~\ref{p_dovoljnomakskonuseposmatrati}(ii), it is enough to show the following: if $M$ is the normal cone to $S_{i-1}$ at a vertex not contained in $F$, i.e.\ if one of the element of $\{-a_j \mid j \in [p]\}$ is not a spanning ray of $M$, then $M$ is contained in $N_{v_j}$ for some $j\in[p]$, which entails that its intersection with $N_{v_j}$ is $M$ itself while the other intersections are not maximal cones; otherwise, i.e.\ if $\{-a_j \mid j \in [p]\}$ is a subset of the spanning set of $M$, then for every $j\in [p]$ the intersection $N\cap N_{v_j}$ is an $n$-cone with $a_\beta$ as a spanning ray, and moreover, the union of all these cones is $M$.

For $j\in[p]$, let $N_{v_j}$ be spanned by the set $\{a_\beta\} \cup \{-a_i \mid i\in [n]-\{j\}\}$.
Recall again that $N_u(S_0)$ is the union of maximal normal cones to $S_{i-1}$ according to Lemma~\ref{l_konusiparalelnetrunkacije}. Namely, the other faces of $S_0$ containing $u$ might be truncated at the previous steps, where the truncation in the vertex $u$ was the first of them. Assume that $h_1$ is an outward normal to the halfspace of that truncation. By Lemma~\ref{l_konusiparalelnetrunkacije}, that truncation produced $n$ maximal normal cones $N_1,\ldots,N_n$ whose union is $N_u(S_0)$, and then, other truncations refined these cones further. We may assume that $N_j$, $j \in [n]$, is spanned by the set $\{h_1\} \cup\{-a_i\mid i\in[n]-\{j\}\}$. Then, applying Remark~\ref{r_haovi} and Lemma~\ref{l_normalezaglavnut} for $I=[p]$ and $J=[n]$, we obtain that $h_1$ is contained in the cone spanned by the set $\{-a_{p+1},-a_{p+2},\ldots,-a_n,a_{\beta}\}$, which entails that for every $j\in[p]$ each spanning ray of $N_j$ is contained in $N_{v_j}$. Hence, $N_j \subseteq N_{v_j}$, i.e.\ every maximal normal cone to $S_{i-1}$ contained in $N_j$ is contained in $N_{v_j}$. 
%Taj konus je u prethodnom nizu paralelnih trunkacijoa podeljeno na uniju $n$ maks konusa (svakom je tu $a_{\beta_i}$)
%(Ako medju preostalima maks konusima sadrzanim u $N_u$ iz $S_{i-1}$ postoji neki koji nije razapet ni jednim h-om osim inicijalnog $h_1$, onda:)Ako je neki od njih maks konus $N$ u $S_{i-1}$ tj. odgovara temenu koje pripada $F$, to znaci da nije nakon trunkiranja $u$ u prethodnim iteracijama trunkirano nista sto to teme sadrzi. Presek takvog konusa sa svakim od $N_{v_p}$ je $n$-konus. Njihova unija je $N$ (to sledi iz toga sto je $a_{\beta}$ simetrala ovih $a_1,...a_p$). Naime, dokaz: $N$ je razapet sa $h_1$ i nekih $n-1$ od $a_1,...,a_n$ (izbacen je neki posle $a_p$). Njegov presek sa $N_{v_j}$ (sve u $N_{v_p}$ je kao kod $N$ ali nema $h_1$ i nema jedan $a_j$,$j\leqslant p$, a ima $a_{\beta}$ i $a_j$ koji ovaj drugi nije imao). Presek im je $a_1,...,a_{p-1}a_{\beta}a_{p+1}...a_nh_1$
Now, if $M$ is one of the remaining maximal cone to $S_{i-1}$ contained in $N_u(S_0)$, then $M$ is spanned by the set obtained from $\{-a_j \mid j \in [n]\}$ by replacing $q$ elements with some vectors $h_1,\ldots,h_q$, such that each of them is an outward normal to the corresponding truncation halfspace. All these truncations were made at some of the previous steps in a face of $S_0$ contained in $F'$.  
%i.e.\ all these truncations were in the faces that correspond to the faces $F_1,\ldots, F_q$ of $S_0$ where $F_1\subset \ldots \subset F_q \subset F'$.
Moreover, all these faces (as well as the corresponding elements of $\mathcal{A}_2$) are mutually comparable. Therefore, by Remark~\ref{r_haovi}, assuming that $h_1=-(a_1+\ldots+a_n)$ and that $h_q$ corresponds to the face witch contains all the others, we conclude that the conditions of Lemma~\ref{l_normalezaglavnut}(ii) are satisfied. Now, we have two cases. Firstly, let us assume that $M$ corresponds to a vertex of $S_{i-1}$ not contained in $F$. Then, for some $j \in [p]$ the ray $-a_j$ is not a spanning ray of $M$. By applying Lemma~\ref{l_normalezaglavnut}(ii) $q-1$ times, we replace the vectors $h_1,\ldots,h_{q-1}$ by the corresponding elements $-a_i$, $p<i\leqslant n$, and obtain that $M$ is contained in the $n$-cone $M'$ spanned by the set $\{h_q\}\cup\{-a_j\mid j \in[n]-\{r\}\}$ for some $r\in[p]$. Since the cone $N_{v_r}$ is spanned by the set \mbox{$\{a_\beta\}\cup\{-a_j\mid j \in[n]-\{r\}\}$}, applying Lemma~\ref{l_normalezaglavnut}(i), we conclude that $M'\subseteq N_{v_r}$, and hence, $M\subseteq N_{v_r}$.  
Otherwise, i.e.\ if $M$ corresponds to some vertex of $S_{i-1}$ contained in $F$, then for every $j \in [p]$, $-a_j$ is a spanning ray of $M$. By Remark~\ref{r_haovi}, $M$ is the union of $n$-cones $N_1,\ldots,N_p$ such that the spanning set of $N_{j'}$, $j'\in[p]$, can be obtained from the spanning set of $M$ by replacing the ray $-a_{j'}$ with the ray $a_\beta$. Now, as above, for each $N_{j'}$ we apply Lemma~\ref{l_normalezaglavnut}(ii) $q$ times. Namely, by replacing the vectors $h_1,\ldots,h_{q}$ with the corresponding elements $-a_i$, where $p<i\leqslant n$, we obtain that $N_{j'}$ is contained in an $n$-cone spanned by the set \mbox{$\{a_\beta\}\cup\{-a_j\mid j \in[n]-\{j'\}\}$}. Since this set spans $N_{v_{j'}}$, we can conclude that for every $j\in [p]$ the intersection $M \cap N_{v_j}$ is $N_{v_j}$, the normal cone to tr$_FS_{i-1}$ at some vertex contained in the truncation hyperplane, and moreover, the union of all these intersections is $M$.
 
Finally, we can conclude that Proposition~\ref{p_dovoljnomakskonuseposmatrati}(ii) holds. Also, all above, together with Lemma~\ref{l_konusiparalelnetrunkacije}, one can verify that there is no vertex of tr$_FS_{i-1}$ which is not obtained in some of the mentioned intersections, i.e.\ the remaining condition is also satisfied. It remains to apply Proposition~\ref{p_dovoljnomakskonuseposmatrati} concluding that $S_{i} \simeq$ tr$_FS_{i-1}$, i.e.\ $Q_{i}$ is a truncator summand for $S_{i-1}$. 

For every $i\in [m]$, at the end of \textit{i}th step, we label facets of $S_i$ in the following manner: the corresponding facets of $S_i$ and $S_{i-1}$ are equilabelled, while the new appeared facet is labelled by $x(i)$ (see Remark~\ref{r_correspondingfacettruncation}). 
At the end, we get $n$-polytope $PA_{n,1}$ as the last obtained sum $S_{m}$. Since for every $i\in[m]$, $Q_{i}$ is a truncator summand for $S_{i-1}$, Definition~\ref{d_Mink_realizacija}(iii) is satisfied. Also, every element of $\mathcal{B}_1$ is used as label for a facet of $PA_{n,1}$ such that equilabelled facets of $PA_{n,1}$ and $\mathbf{PA}_n$ are parallel. This, together with Corollary~\ref{c_labele_odPA_n}, implies $PA_{n,1} \simeq \mathbf{PA}_n$, and hence, Definition~\ref{d_Mink_realizacija}(i) is also satisfied.
\end{proof}

By Corollary~\ref{prop_minkowski_svojstva2}(i), $^M PA_2=PA_{2,1}$ holds, up to translation. Applying Lemma~\ref{l_glavnal} and following the proof of the previous theorem, we obtain the following family of $n$-dimensional Minkowski-realisations of the simplicial complex~$C$.
\begin{thm}\label{t_glavnatfamilija}
For $n\geqslant 2$ and $c\in (0,1]$, the polytope
\[
PA_{n,c}= \Delta_{[n+1]}+\sum_{\beta\in \mathcal{A}_1} \Delta_{\bigcup\beta}+ \sum_{\beta\in \mathcal{A}_2} \bigl (P_{\mathcal{B}_\beta}\bigcap {\pi_{\beta,c}}^\geqslant \bigr )
\]
is an $n$-dimensional Minkowski-realisation of the simplicial complex $C$, which is normally equivalent to $\mathbf{PA}_n$.
\end{thm} 

\begin{center}\textmd{Acknowledgements} 
\end{center}
\medskip I thank Zoran Petri\' c for support and very helpful discussions during my entire research. The paper is also supported by the Grants $III44006$ of
the Ministry for Education and Science of the Republic of Serbia.

{\small
\noindent author: Jelena Ivanovi\' c\\
address: University of Belgrade, Faculty of Architecture\\ Bulevar Kralja Aleksandra 73/II\\ 11000 Belgrade, Serbia\\
e-mail: jelena.ivanovic@arh.bg.ac.rs
}
\end{document}